\documentclass[a4paper,11pt]{amsproc}%{amsproc}%{article}%

\usepackage[top=3.0cm, bottom=3.0cm, inner=3.0cm, outer=3.0cm, includefoot]{geometry}

\usepackage[backend=bibtex,firstinits]{biblatex}
\renewbibmacro{in:}{}

\bibliography{Bib_Liu_Muench_Peyerimhoff_New}{}

\usepackage[latin1]{inputenc}
\usepackage[T1]{fontenc}
\usepackage{verbatim}

\usepackage{geometry}
\usepackage{amssymb}
\usepackage{amsmath}
\usepackage{graphicx}
\usepackage{amsthm}

\usepackage{color}

\usepackage{enumerate}

\usepackage{tikz}

\usepackage{tkz-graph}

\tikzstyle{vertex}=[circle, draw, inner sep=0pt, minimum size=6pt]
\newcommand{\vertex}{\node[vertex]}

\usepackage{tikz}
\usetikzlibrary{arrows, decorations.markings}

\tikzstyle{vecArrow} = [thick, decoration={markings,mark=at position
	1 with {\arrow[semithick]{open triangle 60}}},
double distance=1.4pt, shorten >= 5.5pt,
preaction = {decorate},
postaction = {draw,line width=1.4pt, white,shorten >= 4.5pt}]

\tikzstyle{doubleLine} = [thick, 
double distance=1.4pt, %shorten >= 5.5pt,
preaction = {decorate},
postaction = {draw,line width=1.4pt, white,shorten >= 4.5pt}]

\tikzstyle{innerWhite} = [semithick, white,line width=1.4pt, shorten >= 4.5pt]

	\tikzstyle{RedArrow} = [thick, decoration={markings,mark=at position
	5.5pt with {\arrow[semithick, color=\colfill]{triangle 60 reversed}}},
	decoration={markings,mark=at position
	5.5pt with {\arrow[semithick, color=\colline]{open triangle 60 reversed}}},
	double distance=1.4pt,
	shorten <= 5.5pt,
	preaction = {decorate},
	postaction = {draw,line width=1.4pt, \colfill,shorten >= 4.5pt, shorten <= 4.5pt},
	color=\colline
	]

\newcommand{\minheight}{1.1cm}

\newcommand{\boxpt}{0.06cm}

\newcommand{\boxmid}{10.5cm}

\newcommand{\boxdist}{2.95cm}

\newcommand{\boxwidth}{2.4cm}

\newcommand{\HSSwidth}{1.6cm}

\newcommand{\boxbetween}{\boxdist-\boxwidth}

\newcommand{\withbig}{\boxdist+\boxwidth/2+ 2*\boxpt + \HSSwidth + \boxbetween}

\newcommand{\yih}{\y/2 +\yi/2}

\newcommand{\ybetween}{0.25cm}
\newcommand{\y}{15cm}
\pgfmathsetmacro{\yi}{\y - 2*\ybetween-\minheight}
\newcommand{\colfill}{white}
\newcommand{\colline}{black}

\definecolor{mid-gray}{gray}{0.6}
\definecolor{light-gray}{gray}{0.8}

\newcommand{\D}{\Deg_{\max}}
\newcommand{\IR}{{\mathbb{R}}}

\newcommand{\IN}{{\mathbb{N}}}

\newcommand{\eChar}{\begin{enumerate}[(i)]}
\newcommand{\eCharR}{\begin{enumerate}[(a)]}
\newcommand{\eBr}{\begin{enumerate}[(1)]}

\newcommand{\Deg}{\operatorname{Deg}}
\newcommand{\vol}{\operatorname{vol}}
\newcommand{\diam}{\operatorname{diam}}

\newcommand{\iP}{\mathcal P}
\newcommand{\Abstract}

\usepackage{amsthm}

\makeatletter
\newtheorem*{rep@theorem}{\rep@title}
\newcommand{\newreptheorem}[2]{%
	\newenvironment{rep#1}[1]{%
		\def\rep@title{#2 \ref{##1}}%
		\begin{rep@theorem}}%
		{\end{rep@theorem}}}
\makeatother

\newreptheorem{theorem}{Theorem}
\newreptheorem{lemma}{Lemma}

\title
{
Rigidity properties of the hypercube via Bakry-\'Emery curvature
}

\author{Shiping Liu, Florentin M\"unch, Norbert Peyerimhoff}
\date{\today}

\theoremstyle{plain}
\newtheorem{lemma}{Lemma}[section]
\newtheorem{theorem}[lemma]{Theorem}

\newtheorem{proposition}[lemma]{Proposition}
\newtheorem{corollary}[lemma]{Corollary}

\theoremstyle{definition}

\newtheorem{example}[lemma]{Example}

\newtheorem{rem}[lemma]{Remark}
\newtheorem{defn}[lemma]{Definition}

\numberwithin{equation}{section}

\begin{document}

\maketitle

%\tableofcontents

%\section*{Preface}

\pagestyle{plain}

\begin{abstract}
  We give rigidity results for the discrete Bonnet-Myers diameter bound
  and the Lichnerowicz eigenvalue estimate. Both inequalities are sharp if
  and only if the underlying graph is a hypercube.  The proofs use
  well-known semigroup methods as well as new direct methods which
  translate curvature to combinatorial properties.  Our results can be
  seen as first known discrete analogues of Cheng's
  and Obata's rigidity theorems.
\end{abstract}

\tableofcontents

%\begin{comment}

\section{Introduction}

The hypercube is a well studied object and a variety of combinatorial characterizations have been established. For a survey article on combinatorial properties of the hypercube, see \cite{harary1988survey}. 
We want to point out two particular hypercube characterizations in the literature.
One goes back to Foldes. % (see \cite{foldes1977characterization}).
\begin{theorem}[see \cite{foldes1977characterization}]
An unweighted graph $G$ is a hypercube if and only if
\begin{itemize}
\item
$G$ is bipartite and
\item For all vertices $x,y$, the number of shortest paths between $x$ and $y$ is $d(x,y)!$.
\end{itemize}
\end{theorem}

The other hypercube characterization has been found by Laborde and Hebbare.
\begin{theorem}[see \cite{laborde1982another}]
An unweighted graph $G$ is a hypercube if and only if
\begin{itemize}
\item
$\#V = 2^{\Deg_{\min}}$ and
\item Every pair of adjacent edges is contained in a 4-cycle.
\end{itemize}
\end{theorem}

Another question one might ask is whether the hypercube is already uniquely determined by its local structure. In particular, one might conjecture that every bipartite, regular graph with  all two-balls isomorphic to the hypercube two-ball, needs to be the hypercube. However, this has been disproven by Labborde and Hebbare by the example given in Figure~\ref{fig:NoChar} (see \cite{laborde1982another}).

	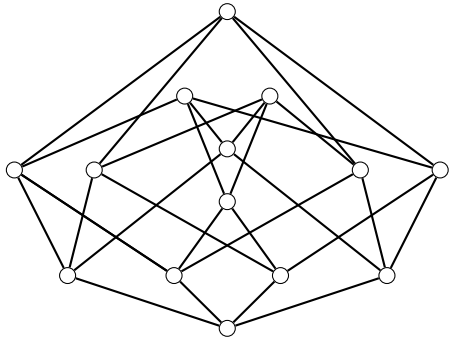
\begin{figure}[h]
	%\centering
	\begin{tikzpicture}[scale=0.7]
	\vertex(x) at (0,0) {};
	\vertex(y1) at (-3,1) {};
	\vertex(y2) at (-1,1) {};
	\vertex(y3) at (1,1) {};
         \vertex(y4) at (3,1) {};	
	 \vertex(z1) at (-4,3) {};	         
	 \vertex(z2) at (-2.5,3) {};
	 \vertex(z3) at (0,2.4) {};
	 \vertex(z4) at (0,3.4) {};
	 \vertex(z5) at (2.5,3) {};
	 \vertex(z6) at (4,3) {};	 	 	 	          
	 \vertex(w1) at (0,6) {};	 	 	 	          
	 \vertex(w2) at (-0.8,4.4) {};	 	 	 	          
	 \vertex(w3) at (0.8,4.4) {};	 	 	 	          	 	 	 
	\Edge(x)(y1)	
	\Edge(x)(y2)
	\Edge(x)(y3)
	\Edge(x)(y4)
	\Edge(y1)(z1)
	\Edge(y1)(z2)
	\Edge(y1)(z4)
	\Edge(y2)(z1)				
	\Edge(y2)(z3)
	\Edge(y2)(z1)
	\Edge(y2)(z5)
	\Edge(y3)(z2)
	\Edge(y3)(z3)
	\Edge(y3)(z6)					
	\Edge(y4)(z4)
	\Edge(y4)(z5)
	\Edge(y4)(z6)	
	\Edge(z1)(w1)
	\Edge(z1)(w2)
	\Edge(z2)(w1)
	\Edge(z2)(w3)				
	\Edge(z3)(w2)
	\Edge(z3)(w3)
	\Edge(z4)(w2)
	\Edge(z4)(w3)
	\Edge(z5)(w1)
	\Edge(z5)(w3)					
	\Edge(z6)(w1)
	\Edge(z6)(w2)

	\end{tikzpicture}
	\caption{The illustrated graph is bipartite, 4-regular and locally isomorphic to the hypercube in the sense of two-balls.}
\label{fig:NoChar}

\end{figure}
%recently attracted interest
The hypercube characterization we present in this paper is completely different in spirit. Our approach is inspired by Riemannian geometry.
On Riemannian manifolds, Ricci curvature is a highly fruitful concept to deduce many interesting analytic and geometric properties like Li-Yau inequality, parabolic Harnack inequality and eigenvalue estimates like Buser inequality.
Assuming a positive lower Ricci-curvature bound yields eminently strong implications. One of them is Myers' diameter bound stating that a complete, connected $n$-dimensional manifold with Ricci-curvature at least a positive constant $K>0$ has a diameter smaller or equal than the $n$-dimensional sphere with Ricci-curvature $K$ (see \cite{myers1941riemannian}). 
The other implication we are interested in this article is the Lichnerowicz eigenvalue bound. It states that if the Ricci-curvature is larger than a positive constant $K>0$, then
one can lower bound the first non-zero eigenvalue of the Laplace-Beltrami operator by $\frac{nK}{n-1}$.
Impressive rigidity results have been found by Cheng (\cite{cheng1975eigenvalue}) and Obata (\cite{obata1962certain}), respectively. 
They have proven that  rigidity of the diameter bound as well as rigidity of the Lichnerowicz eigenvalue estimate can only be attained on the round sphere.

A remarkable analogy between the round sphere $S_N$ and the hypercube $H_N$ is that in both cases, the concentration of measure converges to the Gaussian measure when taking the dimension to infinity.
By concentration of measure we mean a measure $C$ on $[0,\infty)$ given by $C_{S_N}(A) := \vol(x \in S_N : d(x,x_0) \in A)$ for a fixed $x_0 \in S_N$ and $C_{H_N}(A) := \vol(x \in H_N : d(x,x_0) \in A)$ for a fixed $x_0 \in H_N$ when taking the natural
volume measure $\vol$ and distance $d$  on $S_N$ and $H_N$.
Taking a suited normalization yields convergence in distribution of $C_{S_N}$ and $C_{H_N}$ to the Gaussian measure $C_G$ with density $C_G(dx) = e^{-x^2}$.
For details, see e.g. \cite{gromov2007metric,ollivier2010survey}.
This analogy between the round sphere and the hypercube motivates the question whether rigidity properties similar to Cheng's and Obata's sphere theorems hold true for the hypercube. In this paper, we positively answer this question.

%The results presented in this paper suggest strongly that the discrete analogue of the $n$-dimensional round sphere is the $n$-dimensional hypercube.

While theory of Riemannian manifolds is understood very well, the era of computer science demands for discrete objects instead of continuous manifolds. Graphs were introduced as a discrete setting to approximate the behavior of manifolds.
This was the birth of discrete differential geometry. 
According to classical differential geometry, there are various approaches to study curvature and Ricci-curvature in particular. We mention the coarse Ricci-curvature by Ollivier using Wasserstein-metrics \cite{ollivier2009ricci}, the Ricci-curvature via convexity of the entropy by Sturm \cite{sturm2006i,sturm2006ii}, Lott, Villani \cite{lott2009ricci}, and the Bakry-\'Emery-Ricci-curvature \cite{bakry1985diffusions}.
    When explaining curvature of manifolds, the canonical examples are the sphere for positive, the Euclidean plane for zero, and the hyperbolic space for negative curvature.
Related examples can also be given on graphs. These are hypercubes for positive, lattices for zero and trees for negative curvature. In a certain sense, the meaningfulness of a discrete curvature notion can be measured via these examples.
Indeed, the question of the Ricci-curvature of the hypercube has recently attracted interest among several mathematical communities (see 
\cite{erbar2012ricci,gozlan2014displacement,gromov2007metric,klartag2015discrete,ollivier2012curved,villani2016synthetic}) and was asked verbatim by Stroock in a seminar as early as 1998, in a context of logarithmic Sobolev inequalities.
In this article, the hypercube plays one of the leading roles.

The other leading role is played by Bakry's and \'Emery's Ricci-curvature.
Due to Bakry and \'Emery's break through in 1985, a Ricci-curvature notion also became available for discrete settings. Naturally, the question arises whether the strong implications of Ricci-curvature bounds also hold true for graphs.
This is a vibrant topic of recent research and many  results in analogy to manifolds have been established.
     
     We want to particularly point out the discrete version of Myer's diameter bound (see \cite{liu2016bakry} and weaker versions in \cite{fathi2015curvature,horn2014volume}) and Lichnerowicz eigenvalue bound (see e.g. \cite{liu2015curvature,bauer2014curvature}).

\begin{proposition}\label{pro: Intro Myers Lichnerowicz}
Let $G=(V,E)$ be a
simple (i.e., without loops and multiple edges)
connected graph. Let $D$ be the maximal vertex degree.  
Let $\diam(G)$ be the diameter of $G$ w.r.t the combinatorial graph distance.
Let
$0=\lambda_0<\lambda_1 \leq \lambda_2 \leq\ldots$ be the eigenvalues
of the non-normalized Laplacian $-\Delta$, defined in
\eqref{eq:lapunweigh} below. 
	Suppose $G$ satisfies the Bakry-\'Emery curvature-dimension inequality $CD(K,\infty)$. Then,
	\begin{enumerate}
		\item $G$ satisfies Myer's diameter bound, i.e.,
		\begin{align}
		\diam(G)\leq \frac{2D}K.
		\end{align}
		\item $G$ satisfies Lichnerowicz eigenvalue estimate, i.e.,
		\begin{align}
		\lambda_1 \geq K.
		\end{align}
	\end{enumerate}
\end{proposition}

%Cheng proved on manifolds that sharpness of Myers' diameter bound is
%only attained on Euclidean spheres (see \cite{Cheng1975}). Obata
%proved another characterization of the Euclidean sphere via sharpness
%of the Lichnerowicz eigenvalue estimate \textcolor{red}{(see
 % \cite{Obata1962})}.
%\textcolor{red}{References?}

The first assertion follows from \cite[Corollary~2.2]{liu2016bakry}.
The second assertion is the Lichnerowicz spectral gap theorem which can be found in \cite{bauer2014curvature,liu2015curvature} in the graph case.

It is now natural to ask whether analogues of Cheng's and Obata's theorems are still valid on graphs.
This article is dedicated to positively answer this question and to prove that indeed a discrete version of these
rigidity results holds true.  A characterization will be given via the
hypercube which shall be seen as a discrete analogue of the Euclidean sphere.  

For convenience, we first state our main results for unweighted graphs.
%\textcolor{blue}{ This characterization only holds true
%  for unweighted graphs as section~\ref{sec:Examples} demonstrates.}

%\textcolor{blue}{
%For convenience of the reader, we restrict to unweighted graphs at this point.
%}

\begin{theorem} \label{cor:main unweighted} Let $G=(V,E)$ be a
    simple (i.e., without loops and multiple edges)
	connected graph. Let $D$ be the maximal vertex degree.  Let
	$0=\lambda_0<\lambda_1 \leq \lambda_2 \leq\ldots$ be the eigenvalues
	of the non-normalized Laplacian $-\Delta$, defined in
	\eqref{eq:lapunweigh} below. The following are
		equivalent:
	\begin{enumerate}
		\item $G$ is a $D$-dimensional hypercube.  \label{char:hypercube cor}

		\item  $G$ satisfies $CD(K,\infty)$ for some $K>0$ and
			$\lambda_D = {K}$. \label{char:Lichnerowicz cor}
		\item $G$ satisfies $CD(K,\infty)$ for some $K>0$ and
			$\diam_d(G) = \frac{2 D}{K}$. \label{char:diameter bound cor}
		\end{enumerate}
\end{theorem}
	The theorem is a direct consequence of the main theorem (Theorem~\ref{thm:main}) which is concerned with weighted graphs.
\begin{rem}

Theorem~\ref{cor:main unweighted} is connected to the eigenvalue- and diameter bounds from Proposition~\ref{pro: Intro Myers Lichnerowicz} in the following way:
	\begin{itemize} 
		\item 	Statement \ref{char:Lichnerowicz cor} means sharpness of the eigenvalue
		bound $\lambda_D \geq \lambda_1 \geq {K}$ whenever $CD(K,\infty)$ is
		satisfied, see \cite[Theorem~1.6]{liu2015curvature}.
		It is crucial to assume $\lambda_D=K$ and not only $\lambda_1=K$ since the latter is not strong enough to imply that $G$ is the hypercube (see Example~\ref{ex: sharp eigenvalue estimate}).
		However, the  hypercube characterization via $\lambda_D=K$ also holds for weighted graphs without further assumptions.
		
		\item 	Statement \ref{char:diameter bound cor} means sharpness of the diameter
		bound $\diam_d(G) \leq \frac{2 D}{K}$ whenever $CD(K,\infty)$ is
		satisfied (see \cite[Corollary~2.2]{liu2016bakry}).
        To give a hypercube characterization for weighted graphs, we will need to have a further assumption on the uniformity of the edge weight and vertex measure         (see Definition~\ref{def:uniform edge degree}, Section~\ref{sec: const edge deg} and Section~\ref{sec:Examples}).
	\end{itemize}
	%\textcolor{blue}{The corollary is a direct consequence of the main theorem (Theorem~\ref{thm:main}) which treats weighted graphs.
 But before we present our proof strategies and the main theorem for weighted graphs, we explain the organization of the paper and introduce our setup and notations.
\end{rem}

\subsection{Organization of the paper}
In Section~\ref{sec:Concepts}, we introduce our main concepts for exploring sharpness of the $CD$-inequality. In particular in Section~\ref{sec: Hypercube char}, we present our main theorem (Theorem~\ref{thm:main}), i.e., the characterization of the hypercube via curvature sharpness for weighted graphs. 
We give a short proof of our main theorem in this subsection
under assumption of the concepts given until there. All further sections are dedicated to prove the main concepts from Section~\ref{sec:Concepts}.

\subsection{General setup and  notation}
 
Let us start with a rather general definition of a
  graph. A triple $G=(V,w,m)$ is called a
  \emph{(weighted) graph} if $V$ is a countable set, if
$w:V^2=V\times V \to [0,\infty)$ is symmetric and zero on the diagonal and if
$m:V \to (0,\infty)$. We call $V$ the \emph{vertex set}, and $w$ the
\emph{edge weight} and $m$ the \emph{vertex measure}. 
For $x,y\in V$, we write $x\sim y$ whenever $w(x,y)>0$.
We define the
\emph{graph Laplacian} $\IR^V \to \IR^V$ via
\begin{align}
	\Delta f(x) := \frac 1 {m(x)} \sum_y w(x,y)(f(y) - f(x)).
\end{align}  In the
following, we only consider \emph{locally finite} graphs, i.e., for
every $x \in V$ there are only finitely many $y \in V$ with
$w(x,y) >0$.  We write 
\begin{align}\label{eq:defDeg}
\Deg(x) := \frac{\sum_y w(x,y)}{m(x)}
\end{align}
 and
$\D := \sup_x \Deg(x)$.  
Furthermore, we define the \emph{combinatorial vertex degree} $\deg(x):=\#\{y:y\sim x\}$ and $\deg_{\max}:=\sup_x \deg(x)$.
In this article, we will always assume $\D < \infty$ and $\deg(x) < \infty$ for all $x \in V$.
Moreover for
  $A,B \subset V$, we write $\vol(A):=m(A):=\sum_{x\in A} m(x)$
  and $$w(A,B):=\sum_{(x,y)\in A \times B}w(x,y).$$

For some of our rigidity results, we restrict our considerations to unweighted graphs.  
\begin{defn}[Unweighted representation of a graph]\label{def: unweighted Repr}
For a graph $G=(V,w,m)$, we define the set of unoriented edge set $E:=\{\{x,y\}:w(x,y)>0\}$.
We call 
$\widetilde G:=(V,E)$ the \emph{unweighted representation} of $G$.  We call $G=(V,E)$ to be an
\emph{unweighted graph} and we define the \emph{non-normalized Laplacian} as 
\begin{equation} \label{eq:lapunweigh}
\Delta f(x) = \sum_{y\sim x} (f(y)-f(x)).
\end{equation}
If furthermore $w(x,y) \in \{0,1\}$ and
$m(x) = 1$ for all $x,y \in V$, we identify $G$ with $\widetilde G$ since the Laplacians of $G$ and $\widetilde G$ coincide.
Moreover, an unweighted graph $G=(V,E)$ is simple, i.e., it has no
multiple edges by the very construction and $G$ is without loops since we
have $w(x,x)=0$ for all $x \in V$.
\end{defn}

For rigidity results on the diameter, we need uniformity of the edge degree which we define now.

\begin{defn}[Edge degree]\label{def:uniform edge degree}
	Let $G=(V,w,m)$ be a weighted graph. Let $E^{or}:=\{(x,y):x\sim y\}$ be the set of oriented edges, i.e., we distinguish an edge $(x,y)$  from $(y,x)$.
	Additionally to the vertex degrees $\deg$ and $\Deg$, we define the \emph{ edge degree } $\kappa:E^{or} \to \IR_+$ via
	$\kappa(x,y):=w(x,y)/m(x)$.
		We say that $G$ has \emph{ constant edge degree} $\kappa_0$ if $\kappa(x,y) \in \{0,\kappa_0\}$ for all $x,y\in V$.
	%	We write $\kappa=const$ if there exists $\kappa_0$ s.t.
\end{defn}

We remark that the notation $\kappa$ corresponds to a standard notation of Markov kernels, but in our setting, we do not need any normalization property of $\kappa$.

Let us give a
definition of the hypercube which is particularly useful for our
purposes.
\begin{defn}[Hypercube]\label{def:hypercube}
	Let $D \in \IN$ and let $[D]:=\{1,\ldots,D\}$. We denote the
	power set by $\iP$. For $A,B \in \iP([D])$, we
	denote the symmetric difference by
	$A \ominus B := (A \cup B) \setminus (A \cap B)$.  We define
	$E_D := \{\{A,B\} \in \iP([D])\times \iP([D]) : \#(A \ominus B) =1
	\}$. % We now say 
	Then the unweighted graph $H_D=(\iP([D]),E_D)$ is a realisation of a
		\emph{$D$-dimensional hypercube}.
    We say a weighted graph $G=(V,w,m)$ is a $D$-dimensional hypercube if its unweighted representation $\widetilde G$ is a $D$-dimensional hypercube.
\end{defn}

\begin{rem}
	This definition is equivalent to another standard definition of the hypercube, i.e.,
	$H_D  \cong  (\{0,1\}^D,E)$ s.t. $v\sim w$ iff $\|v-w\|_1=1$ for all $v,w \in \{0,1\}^D$.
\end{rem}

\begin{defn} [Bakry-\'Emery-curvature] \label{def:BE-curv} The
  \emph{Bakry-\'Emery-operators} for functions
  $f,g: V \to \IR$ are defined via
  $$
  2\Gamma(f,g) := \Delta (fg) - f\Delta g - g\Delta f
  $$
  and
  $$
  2\Gamma_2(f,g) := \Delta \Gamma(f,g) - \Gamma(f, \Delta g) -
  \Gamma(g,\Delta f).
  $$
  We write $\Gamma(f):= \Gamma(f,f)$ and $\Gamma_2(f):=\Gamma_2(f,f)$.
	
  A graph $G$ is said to satisfy the \emph{curvature dimension
    inequality} $CD(K,n)$ for some $K\in \IR$ and $n\in (0,\infty]$
    at a vertex $x \in V$ if for all $f: V \to \IR$,
  $$
  \Gamma_2(f)(x) \geq \frac 1 n (\Delta f)^2(x) + K \Gamma f(x).
  $$
  $G$ satisfies $CD(K,n)$ (globally), if it satisfies $CD(K,n)$ at
  all vertices.
\end{defn}
We remark $ 2\Gamma(f,g)(x) = \frac{1}{m(x)} \sum_{y \sim x} w(x,y)(f(y)-f(x))(g(y)-g(x))$  for $f,g:V\to \IR$ and $x \in V$.
Therefore, $\Gamma(f) \ge 0$. 
Now we define the combinatorial 
%and resistance 
metric and diameter.

\begin{defn}[Combinatorial metric]\label{def:combinatorial metric}
  Let $G=(V,w,m)$ be a locally finite graph.  We define the
  \emph{combinatorial metric} $d:V^2 \to [0,\infty)$ via
  $$
  d(x,y) := \min \{n : \mbox{ there exist } x=x_0,\ldots,x_n=y \mbox{
    s.t. } w(x_i,x_{i-1})>0, \mbox{ all } i=1\ldots n\}.
  $$
  and the \emph{combinatorial diameter} via
  $\diam_d(G) := \sup_{x,y \in V} d(x,y)$.
\end{defn}

We define the backwards-degree w.r.t. $x_0\in V$ via
$$d_-^{x_0}(z):= \sum_{\stackrel{y\sim
      z}{d(y,x_0)<d(z,x_0)}} \frac{w(y,z)}{m(z)}$$
%\# \{y \sim z: d(y,x)<d(z,x)\}$
and the forwards-degree $$d_+^{x_0}(z):=
  \sum_{\stackrel{y\sim z}{d(y,x_0)>d(z,x_0)}}
  \frac{w(y,z)}{m(z)}.$$ For $A \subset V$, we define $d_A(x) :=
w(\{x\},A)/m(x).$ %\#\{a \in A: a \sim x\}$. 
The sphere and ball of radius $k$ around $x \in V$ are
  defined as $S_k(x):=\{y \in V: d(x,y)=k\}$ and
  $B_k(x):=\{y \in V: d(x,y)\leq k\}$.

\section{Concepts and main results for weighted graphs}\label{sec:Concepts}

In this section, we start considering abstract criteria for sharpness of the $CD$ inequality.
The criteria will be applied to the distance functions which will motivate the notion of a \emph{hypercube shell structure}.
For characterization of diameter sharpness, we moreover need a constant edge degree which essentially means standard weights.
Additionally to the abstract criteria of $CD$ sharpness, we need a combinatorial approach via the \emph{small sphere property} and the \emph{non-clustering property} (see Definition~\ref{def: SSP NCP}) to characterize the hypercube.

\subsection{Abstract curvature sharpness properties}

In our investigations of sharpness of the $CD$ inequality, we start with a basic observation. Suppose a graph $G=(V,w,m)$ satisfies $CD(K,\infty)$, then for all $f \in C(V)$, one has
\begin{enumerate}
	\item $e^{-2Kt} P_t \Gamma f \geq \Gamma P_t f$.
	\item $\Gamma_2 f \geq  K \Gamma f$.
	\item $\lambda_1 \geq K$. 
\end{enumerate} 

%\textcolor{red}{References?}

The first assertion in the manifold case can be found e.g. in \cite[Proposition~3.3]{bakry2004functional}, in \cite[Lemma~5.1]{ledoux2004spectral},   and in \cite[Theorem~1.1]{wang2011equivalent}. For graphs, it can be found e.g. in \cite[Lemma~2.11]{liu2014eigenvalue} and \cite[Theorem~3.1]{lin2015equivalent}.
The second assertion is the definition of $CD(K,\infty)$.
The third assertion is the Lichnerowicz spectral gap theorem which can be found for graphs in \cite{bauer2014curvature}  
and for the more general graph connection Laplacians in \cite{liu2015curvature}.
    Indeed, sharpness of one of the inequalities above implies sharpness of all other ones in a very precise way, as stated in the following theorem which will reappear as Theorem~\ref{thm:abstract CD sharpness} and be proven there.

\begin{figure}[h]
	%	\centering	
	\resizebox{15cm}{!}
	{
		\begin{tikzpicture}[thick]
		
		\tikzstyle{doubleRedArrow} = [thick, decoration={markings,mark=at position
			5.5pt with {\arrow[semithick, color=\colfill]{triangle 60 reversed}}},
		decoration={markings,mark=at position
			5.5pt with {\arrow[semithick, color=\colline]{open triangle 60 reversed}}},
		decoration={markings,mark=at position
			0.99999 with {\arrow[semithick, color = \colfill]{triangle 60}}},
		decoration={markings,mark=at position
			1 with {\arrow[semithick, color = \colline]{open triangle 60}}},
		double distance=1.4pt,
		shorten <= 5.5pt,
		shorten >= 5.5pt,
		preaction = {decorate},
		postaction = {draw,line width=1.4pt, \colfill,shorten >= 4.5pt, shorten <= 4.5pt},
		color=\colline
		]

		\node[draw,rectangle, color=\colline,align=center,text width = 2cm, minimum height = \minheight, minimum width = 2.4cm] (Pt) at (\boxmid-\boxdist,\y){$\Gamma P_t f = e^{-2Kt} P_t \Gamma f$};
		\node[draw, color=\colline,rectangle,align=center,text width = 2cm, minimum height = \minheight, minimum width = 2.4cm] (phi) at (\boxmid,\y){$f = \varphi + C$, $ -\Delta \varphi = K\varphi$};
		\node[draw, color=\colline,rectangle,align=center,text width = 2cm, minimum height = \minheight, minimum width = 2.4cm] (G2) at (\boxmid+\boxdist,\y){$\Gamma_2f= K\Gamma f$};
		
		\renewcommand{\colline}{light-gray}
		
		\node[draw, color=\colline,rectangle,align=center,text width = 6.0cm, minimum height = \minheight, minimum width = 2*\boxdist + \boxwidth] (dist) at (\boxmid,0cm+\yi){$f:=d(x_0,\cdot), \quad \Deg(x_0)=D$};
		\node[draw, color=\colline,rectangle,align=center,text width = 6.0cm, minimum height = \y-\yi + \minheight + 2*\boxpt, minimum width = 2*\boxdist+\boxwidth+ 4*\boxpt] (distbigbox) at (\boxmid,\y/2 +\yi/2 - \boxpt){};
		\node[draw, color=\colline,rectangle,align=center, text width = 1.2cm, minimum height = \y-(\yi) + \minheight + 2*\boxpt, minimum width = \HSSwidth] (HSS) at (\boxmid + \boxdist+\boxwidth/2+ 2*\boxpt + \HSSwidth/2 + \boxbetween,\y/2 +\yi/2 - \boxpt){$HSS$};
		
		%\node[draw,rectangle,align=center, text width = 1.2cm, minimum height = \y- \yi + \minheight + 2*\boxpt, minimum width = \HSSwidth] (dist) at (\boxmid - \boxdist -\boxwidth/2- 2*\boxpt - \HSSwidth/2  -\boxdist +\boxwidth ,\y/2 + \yi/2 - \boxpt){${\lambda_{\deg_{\max}}}$ $= K$};
		
		\node[draw, color=\colline,rectangle,align=center, text width = 1.2cm, minimum height = \y- \yi + \minheight + 2*\boxpt, minimum width = \HSSwidth] (diam) at (\boxmid - \boxdist -\boxwidth/2- 2*\boxpt - \HSSwidth/2  -\boxdist +\boxwidth ,\yih - \boxpt){$d(x_0,y)$ $= \frac{2D}K$};
		
		%\node[draw,rectangle,align=center, text width = 1.2cm, minimum height = \y- \yi + \minheight + 2*\boxpt, minimum width = \HSSwidth] (diam) at (\boxmid - \boxdist -\boxwidth/2- 2*\boxpt - \HSSwidth/2  -\boxdist +\boxwidth ,\yih - \boxpt){$\|d(x_0,\cdot)\|$ $= \frac{2D}K$};

		%\node[draw,rectangle,align=center,text width = 2.0cm, minimum height = \minheight, minimum width = 18cm] (Pt) at (\boxmid,\y + \minheight){$CD(K,\infty)$};

		\node[draw, color=\colline,rectangle,align=center,text width = 2.0cm, minimum height = \minheight, minimum width = 2*(\withbig)] (kappa) at (\boxmid,\y - 2*\minheight - 2*\boxpt-4*\ybetween){$\kappa=const.$};

		\node[draw, color=\colline,rectangle,align=center,text width = 2.0cm, minimum height = 3*\minheight + 4*\ybetween+4*\boxpt, minimum width = 2*(\withbig)+4*\boxpt] (kappabigbox) at (\boxmid,\y - \minheight - 2*\ybetween -2*\boxpt){};
		
		\node[draw, color=\colline,rectangle,align=center,text width = 1.2cm, minimum height = 3*\minheight + 4*\ybetween+4*\boxpt, minimum width = \HSSwidth] (lambda) at (\boxmid - 1.5*\boxwidth - 1.5*\HSSwidth - 3*\boxdist + 3*\boxwidth - 4*\boxpt,\y - \minheight - 2*\ybetween -2*\boxpt){${\lambda_{\deg_{\max}}}$ $= K$};
		
		\renewcommand{\colline}{black}		
		
		\node[draw, color=\colline,rectangle,align=center,text width = 5.0cm, minimum height = \minheight, minimum width = \HSSwidth + \boxdist-\boxwidth + 2*(\withbig)+4*\boxpt ] (CD) at (\boxmid- \boxwidth/2 - \boxdist/2 + \boxwidth/2 + 0.5*\boxwidth - 0.5*\HSSwidth,\y + \minheight){$CD(K,\infty)$ and $x_0 \in V$};
		
		\renewcommand{\colline}{light-gray}	
		
		\node[draw, color=\colline,rectangle,align=center,text width = 2.0cm, minimum height = 4*\minheight + 4*\ybetween + 8*\boxpt, minimum width = \HSSwidth + \boxdist-\boxwidth + 2*(\withbig)+8*\boxpt ] (CDbigbox) at (\boxmid- \boxwidth/2 - \boxdist/2 + \boxwidth/2 + 0.5*\boxwidth - 0.5*\HSSwidth,\y - \minheight/2 - 2*\ybetween - 2*\boxpt){};

		\node[draw, color=\colline,rectangle,align=center,text width = 2.0cm, minimum height = 4*\minheight + 4*\ybetween + 8*\boxpt, minimum width = \boxwidth ] (HC) at (\boxmid + 2*\boxwidth +3*\boxdist -3*\boxwidth + \HSSwidth + 6*\boxpt  ,\y - \minheight/2 - 2*\ybetween - 2*\boxpt){$\frac{2D}{K}$-dim. Hypercube with $\kappa = \frac K 2$};

		\renewcommand{\colline}{black}			
		
		\draw[doubleRedArrow] (Pt.east)  to  node[above] {}  (phi.west); 
		\draw[doubleRedArrow] (G2.west)  to  node[above] {}  (phi.east);

		\renewcommand{\colline}{light-gray}

		\draw[doubleRedArrow] (HC.west)  to node[above] {}  (CDbigbox.east) ; 
		\draw[doubleRedArrow] (lambda.east)  to  node[above] {}  (kappabigbox.west); 
		\draw[doubleRedArrow] (diam.east)  to node[above] {}  (distbigbox.west); 
		\draw[doubleRedArrow] (HSS.west)  to  node[above] {}  (distbigbox.east);

		\end{tikzpicture}
	}
	\caption{A scheme of Theorem~\ref{thm:abstract CD sharpness INTRO}
	}
	\label{fig:Abstract}
\end{figure}
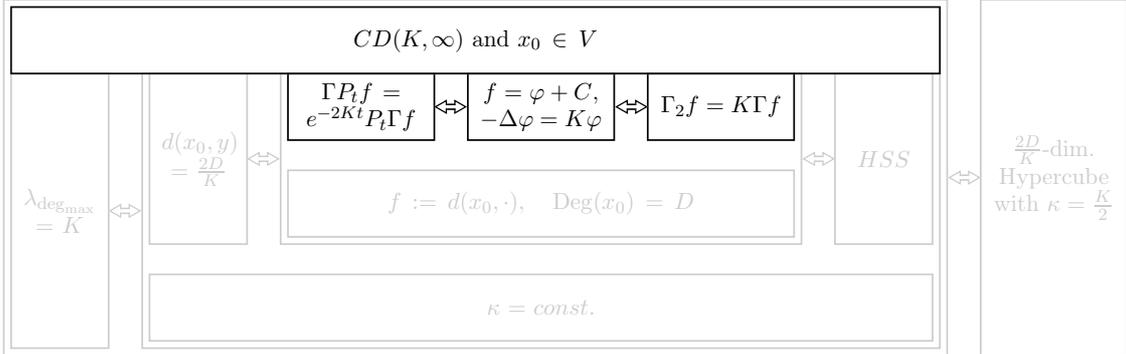

\begin{theorem}[Abstract $CD$-sharpness properties]\label{thm:abstract CD sharpness INTRO}
	Let $G=(V,w,m)$ be a connected graph
	with $\D < \infty$ and 
	satisfying $CD(K,\infty)$.
	%Suppose $\varphi$ to be an eigenfunction to the eigenvalue $K$.
	Let $f\in C(V) $ be a function.
	The following are equivalent.
	\begin{enumerate}
		\item $\Gamma P_t f = e^{-2Kt} P_t \Gamma f$.
		\item $f= \varphi +C$ for a constant $C$ and an eigenfunction $\varphi$ to the eigenvalue $K$ of $-\Delta$.
		\item $\Gamma_2 f = K \Gamma f$.
	\end{enumerate}
	If one of the above statements holds true, we moreover have $\Gamma f = const.$
\end{theorem}
% be proven in section~\ref{sec:sharp CD} on sharp $CD$-inequality.

\subsection{Hypercube shell structure}

Unfortunately, sharp diameter bounds do not imply the graph to be a hypercube in the weighted case (see section~\ref{sec:Examples}). But nevertheless, we can characterize diameter sharpness via a geometric property roughly stating that the graph has the same amount of edges between the spheres as the hypercube. This property is the following.%(see Definition~\ref{def:hypercube shell structure}).

%\textcolor{blue}{
	\begin{defn}[Hypercube shell structure]\label{def:hypercube shell structure}
		We say that a weighted graph $G=(V,w,m)$ has the hypercube shell structure $HSS(N,W,x_0)$ with dimension $N \in (0,\infty)$ and weight $W \in (0,\infty)$ w.r.t. $x_0 \in V$ if\begin{enumerate}
			\item $G$ has constant vertex degree $\Deg(x)=NW$ for all $x\in V$,
			\item $G$ is bipartite,
			\item $d_-^{x_0}(x)=W\cdot d(x,x_0)$ for all $x \in V$.				
		\end{enumerate}
		We say a that graph $G=(V,w,m)$ has the hypercube shell structure $HSS(N,W)$, if there exists $x_0$, s.t. $G$ has the the  hypercube shell structure $HSS(N,W,x_0)$.
	\end{defn}
%}

Intuitively, the hypercube shell structure determines the strength of the connection between vertices at distance $d$ from $x_0$ and shells, i.e., spheres of radius $d-1$ around $x_0$, but not between two certain vertices. 

%\textcolor{blue}{
	\begin{example}
		It is straightforward to confirm that the unweighted $N$-dimensional hypercube has the hypercube shell structure $HSS(N,1,x_0)$ for all $x_0 \in V$.		
	\end{example}
%}

We now state the announced equivalence of diameter sharpness and the hypercube shell structure.% (see $\stackrel{2}{\Leftrightarrow}$, $\stackrel{3}{\Leftrightarrow}$, $\stackrel{4}{\Leftrightarrow}$ and $\stackrel{5}{\Leftrightarrow}$) in Figure~\ref{fig:results} in section \ref{sec: Hypercube char}.

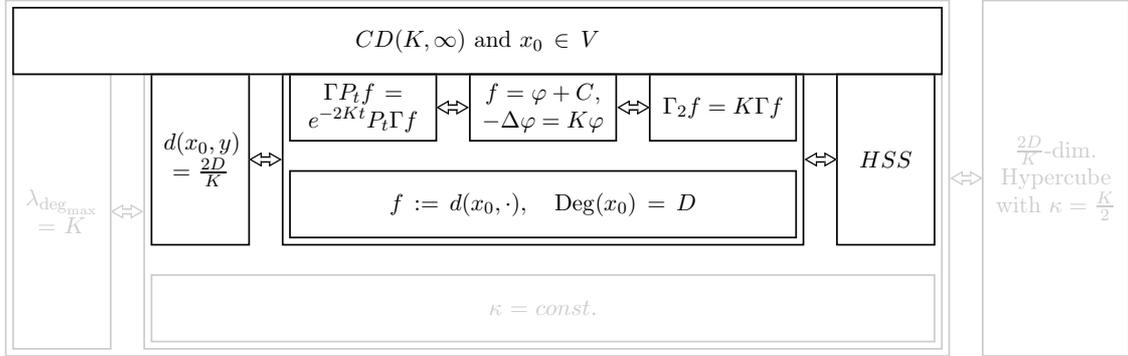
\begin{figure}[h]
	%	\centering	
	\resizebox{15cm}{!}
	{
		\begin{tikzpicture}[thick]
		
		\tikzstyle{doubleRedArrow} = [thick, decoration={markings,mark=at position
			5.5pt with {\arrow[semithick, color=\colfill]{triangle 60 reversed}}},
		decoration={markings,mark=at position
			5.5pt with {\arrow[semithick, color=\colline]{open triangle 60 reversed}}},
		decoration={markings,mark=at position
			0.99999 with {\arrow[semithick, color = \colfill]{triangle 60}}},
		decoration={markings,mark=at position
			1 with {\arrow[semithick, color = \colline]{open triangle 60}}},
		double distance=1.4pt,
		shorten <= 5.5pt,
		shorten >= 5.5pt,
		preaction = {decorate},
		postaction = {draw,line width=1.4pt, \colfill,shorten >= 4.5pt, shorten <= 4.5pt},
		color=\colline
		]

		\renewcommand{\colline}{black}

		\node[draw,rectangle, color=\colline,align=center,text width = 2cm, minimum height = \minheight, minimum width = 2.4cm] (Pt) at (\boxmid-\boxdist,\y){$\Gamma P_t f = e^{-2Kt} P_t \Gamma f$};
		\node[draw, color=\colline,rectangle,align=center,text width = 2cm, minimum height = \minheight, minimum width = 2.4cm] (phi) at (\boxmid,\y){$f = \varphi + C$, $ -\Delta \varphi = K\varphi$};
		\node[draw, color=\colline,rectangle,align=center,text width = 2cm, minimum height = \minheight, minimum width = 2.4cm] (G2) at (\boxmid+\boxdist,\y){$\Gamma_2f= K\Gamma f$};

		\node[draw, color=\colline,rectangle,align=center,text width = 6.0cm, minimum height = \minheight, minimum width = 2*\boxdist + \boxwidth] (dist) at (\boxmid,0cm+\yi){$f:=d(x_0,\cdot), \quad \Deg(x_0)=D$};

		\renewcommand{\colline}{black}
		
		\node[draw, color=\colline,rectangle,align=center,text width = 6.0cm, minimum height = \y-\yi + \minheight + 2*\boxpt, minimum width = 2*\boxdist+\boxwidth+ 4*\boxpt] (distbigbox) at (\boxmid,\y/2 +\yi/2 - \boxpt){};
		\node[draw, color=\colline,rectangle,align=center, text width = 1.2cm, minimum height = \y-(\yi) + \minheight + 2*\boxpt, minimum width = \HSSwidth] (HSS) at (\boxmid + \boxdist+\boxwidth/2+ 2*\boxpt + \HSSwidth/2 + \boxbetween,\y/2 +\yi/2 - \boxpt){$HSS$};
		
		%\node[draw,rectangle,align=center, text width = 1.2cm, minimum height = \y- \yi + \minheight + 2*\boxpt, minimum width = \HSSwidth] (dist) at (\boxmid - \boxdist -\boxwidth/2- 2*\boxpt - \HSSwidth/2  -\boxdist +\boxwidth ,\y/2 + \yi/2 - \boxpt){${\lambda_{\deg_{\max}}}$ $= K$};
		
		\node[draw, color=\colline,rectangle,align=center, text width = 1.2cm, minimum height = \y- \yi + \minheight + 2*\boxpt, minimum width = \HSSwidth] (diam) at (\boxmid - \boxdist -\boxwidth/2- 2*\boxpt - \HSSwidth/2  -\boxdist +\boxwidth ,\yih - \boxpt){$d(x_0,y)$ $= \frac{2D}K$};
		
		%\node[draw,rectangle,align=center, text width = 1.2cm, minimum height = \y- \yi + \minheight + 2*\boxpt, minimum width = \HSSwidth] (diam) at (\boxmid - \boxdist -\boxwidth/2- 2*\boxpt - \HSSwidth/2  -\boxdist +\boxwidth ,\yih - \boxpt){$\|d(x_0,\cdot)\|$ $= \frac{2D}K$};

		%\node[draw,rectangle,align=center,text width = 2.0cm, minimum height = \minheight, minimum width = 18cm] (Pt) at (\boxmid,\y + \minheight){$CD(K,\infty)$};

		\renewcommand{\colline}{light-gray}
		
		\node[draw, color=\colline,rectangle,align=center,text width = 2.0cm, minimum height = \minheight, minimum width = 2*(\withbig)] (kappa) at (\boxmid,\y - 2*\minheight - 2*\boxpt-4*\ybetween){$\kappa=const.$};

		\node[draw, color=\colline,rectangle,align=center,text width = 2.0cm, minimum height = 3*\minheight + 4*\ybetween+4*\boxpt, minimum width = 2*(\withbig)+4*\boxpt] (kappabigbox) at (\boxmid,\y - \minheight - 2*\ybetween -2*\boxpt){};
		
		\node[draw, color=\colline,rectangle,align=center,text width = 1.2cm, minimum height = 3*\minheight + 4*\ybetween+4*\boxpt, minimum width = \HSSwidth] (lambda) at (\boxmid - 1.5*\boxwidth - 1.5*\HSSwidth - 3*\boxdist + 3*\boxwidth - 4*\boxpt,\y - \minheight - 2*\ybetween -2*\boxpt){${\lambda_{\deg_{\max}}}$ $= K$};
		
		\renewcommand{\colline}{black}		
		
		\node[draw, color=\colline,rectangle,align=center,text width = 5.0cm, minimum height = \minheight, minimum width = \HSSwidth + \boxdist-\boxwidth + 2*(\withbig)+4*\boxpt ] (CD) at (\boxmid- \boxwidth/2 - \boxdist/2 + \boxwidth/2 + 0.5*\boxwidth - 0.5*\HSSwidth,\y + \minheight){$CD(K,\infty)$ and $x_0 \in V$};
		
		\renewcommand{\colline}{light-gray}	
		
		\node[draw, color=\colline,rectangle,align=center,text width = 2.0cm, minimum height = 4*\minheight + 4*\ybetween + 8*\boxpt, minimum width = \HSSwidth + \boxdist-\boxwidth + 2*(\withbig)+8*\boxpt ] (CDbigbox) at (\boxmid- \boxwidth/2 - \boxdist/2 + \boxwidth/2 + 0.5*\boxwidth - 0.5*\HSSwidth,\y - \minheight/2 - 2*\ybetween - 2*\boxpt){};

		\node[draw, color=\colline,rectangle,align=center,text width = 2.0cm, minimum height = 4*\minheight + 4*\ybetween + 8*\boxpt, minimum width = \boxwidth ] (HC) at (\boxmid + 2*\boxwidth +3*\boxdist -3*\boxwidth + \HSSwidth + 6*\boxpt  ,\y - \minheight/2 - 2*\ybetween - 2*\boxpt){$\frac{2D}{K}$-dim. Hypercube with $\kappa = \frac K 2$};

		\renewcommand{\colline}{black}			
		
		\draw[doubleRedArrow] (Pt.east)  to  node[above] {}  (phi.west); 
		\draw[doubleRedArrow] (G2.west)  to  node[above] {}  (phi.east);

		\renewcommand{\colline}{light-gray}

		\draw[doubleRedArrow] (HC.west)  to node[above] {}  (CDbigbox.east) ; 
		\draw[doubleRedArrow] (lambda.east)  to  node[above] {}  (kappabigbox.west); 
		
		\renewcommand{\colline}{black}
		
		\draw[doubleRedArrow] (diam.east)  to node[above] {}  (distbigbox.west); 
		\draw[doubleRedArrow] (HSS.west)  to  node[above] {}  (distbigbox.east);

		\end{tikzpicture}
	}
	\caption{This is a scheme of Theorem~\ref{thm:weighted Cheng INTRO}. The box $HSS$ is an abbreviation for the hypercube shell structure $HSS(\frac{2D}K,\frac K 2, x_0)$.
	}
	\label{fig:Cheng}
\end{figure}

\begin{theorem}[Diameter sharpness for weighted graphs]\label{thm:weighted Cheng INTRO}
	Let $G=(V,w,m)$ be a connected graph satisfying $CD(K,\infty)$ for some $K>0$. Let $x_0 \in V$ and let $f_0 := d(x_0,\cdot)$. Suppose $D:=\D<\infty$. The following are equivalent:
	\begin{enumerate}
		\item There exists $y\in V$ s.t. $d(x_0,y) = \frac {2D}K$.
		\item $\Deg(x_0)=D$ and $\Gamma P_t f_0 = e^{-2Kt} P_t \Gamma f_0$.
		\item $\Deg(x_0)=D$ and $f_0= \varphi +C$ for a constant $C$ and an eigenfunction $\varphi$ to the eigenvalue $K$ of $-\Delta$.
		\item $\Deg(x_0)=D$ and $\Gamma_2 f_0 = K \Gamma f_0$.
		\item $G$ has the  hypercube shell structure $HSS\left(\frac{2D}K,\frac K 2,x_0 \right)$.
	\end{enumerate}
\end{theorem} 
The theorem will reappear as Theorem~\ref{thm:weighted Cheng}.

%The proof can be found in section~\ref{sec:Sharp diameter bounds}.

Indeed, there are graphs apart from the hypercube with hypercube shell structure $HSS\left(\frac{2D}K,\frac K 2,x_0 \right)$ satisfying $CD(K,\infty)$.
Examples are given in Corollary~\ref{cor:HSS but no hypercube}.

Based on the theorem, it seems natural to ask whether $HSS$ by itself already implies positive curvature. But this turns out to be false (see Example~\ref{Ex:HSS no CD}).

The hypercube shell structure already determines the volume growth of the graph.
\begin{proposition}
	Let $G=(V,w,m)$ be a weighted graph satisfying $HSS(N,W,x_0)$ for some $x_0 \in V$. Then,
	$$m(S_n(x_0)) = m(x_0)\cdot  {N \choose n}.$$
\end{proposition}

\begin{proof}
We first remark that by bipartiteness, one has $d_-^{x_0}(y) + d_+^{x_0}(y) = \Deg(y)$ for all $y \in V$.	
Therefore, the hypercube shell structure $HSS(N,W,x_0)$ implies
\begin{align*}
W\cdot m(S_k(x_0))(N - k) &= m(S_k(x_0))(\D  - kW)\\&=\sum_{y \in S_{k}(x_0)} m(y)d_+^{x_0}(y) \\&=	w(S_k(x_0),S_{k+1}(x_0)) \\&= \sum_{z \in S_{k+1}(x_0)} m(z)d_-^{x_0}(z)
\\&= m(S_{k+1}(x_0)) W (k+1).
\end{align*}
Hence,
$$\frac{m(S_{k+1}(x_0))}{m(S_k(x_0))}= \frac{N-k}{k+1}$$
which implies $m(S_k(x_0)) = m(x_0){N \choose k}$ via
induction. This finishes the proof.
\end{proof}

\subsection{Constant edge degree}\label{sec: const edge deg}

To characterize the hypercube, and not only the hypercube shell structure via diameter sharpness, we need a further assumption on the uniformity of the edge weight and vertex measure.  This assumption is the constancy of the edge degree (see Definition~\ref{def:uniform edge degree}).

We give a very basic characterization of constant edge degree which will be our further assumption to characterize the hypercube via diameter sharpness. One characterization refers to the unweighted representation which was defined in Definition~\ref{def: unweighted Repr}.
\begin{lemma}\label{l:uniform edge degree basic}
	Let $G=(V,w,m)$ be a weighted connected graph.  Let $\Delta$ be the Laplacian corresponding to $G$ and let $\widetilde \Delta$ be the Laplacian corresponding to the unweighted representation $\widetilde G$ of $G$.	Let $\kappa_0>0$.
	The following are equivalent.\begin{enumerate}
		\item G has constant edge degree $\kappa_0$.
		\item $m(x)=m_0=const$ and $w(x,y) \in \{0,\kappa_0 m_0\}$.
		\item $\Delta = \kappa_0 \widetilde \Delta$.
	\end{enumerate}
\end{lemma}
\begin{proof}
	Implications $2 \Rightarrow 3$  and $3 \Rightarrow 1$ are trivial.
	For proving $1 \Rightarrow 2$, we observe that $\kappa(x,y)=\kappa(y,x)=\kappa_0$ for $x \sim y$. This directly implies $m(x)=m(y)$. Since $G$ is connected, $m$ must be constant on $V$ which easily implies $w(x,y) \in \{0,\kappa_0 m_0\}$.
\end{proof}

In the second assertion of the lemma, we see that a graph $G$ with constant edge degree can be considered as a scaled variant of the unweighted representation $\widetilde G$ of $G$.
We now investigate the compatibility between the scaling behavior of the edge degree, the curvature dimension inequality $CD$ and the hypercube shell structure $HSS$.

	\begin{lemma}\label{l:uniform weights scaling}
		Let $G=(V,w,m)$ be a graph with constant edge degree $\kappa_0$. Let $K \in \IR$ and $n,D>0$ and let $x_0 \in V$. Then,
		\begin{enumerate}
			\item [(i)]$G$ satisfies $CD(\kappa_0 K,n)$ if and only if
			$\widetilde G$ satisfies $CD(K,n)$.
			\item [(ii)] $G$ has the  hypercube shell structure $HSS(D,W,x_0)$ if and only if $W=\kappa_0$ and $\widetilde G$ has the  hypercube shell structure $HSS(D,1,x_0)$.
		\end{enumerate}
	\end{lemma}
	\begin{proof}
		The  first assertion of the lemma easily follows from the fact that a graph $G$ with constant edge degree is a scaled version of its unweighted representation $\widetilde G$ and from the scaling behavior of the curvature dimension condition $CD$.
		
		We finally prove the second assertion.
		Assume $G$ satisfies $HSS(D,W,x_0)$. %Then, there exists $x_0$ s.t. for all $y \sim x_0$,
		Then for all $y\sim x_0$, one has
		\begin{align*}
		\kappa_0=\frac{w(x_0,y)}{m(y)}= d_-^{x_0}(y) = Wd(x_0,y)= W.
		\end{align*}
		This easily implies that $\widetilde G$ satisfies $HSS(D,1)$.
		Vice versa, if $\widetilde G$ satisfies $HSS(D,1)$ and if $G$ has constant edge degree $W=\kappa_0$, then it is straight forward to see that $G$ satisfies $HSS(D,W)$.
	\end{proof}

If we want to characterize the hypercube via diameter sharpness, we need to assume a constant edge degree. Surprisingly, if, in contrast, we want to characterize the hypercube via eigenvalue sharpness, we get the hypercube shell structure $HSS$ and a constant edge degree for free:

\begin{figure}[h]
	%	\centering	
	\resizebox{15cm}{!}
	{
		\begin{tikzpicture}[thick]
		
		\tikzstyle{doubleRedArrow} = [thick, decoration={markings,mark=at position
			5.5pt with {\arrow[semithick, color=\colfill]{triangle 60 reversed}}},
		decoration={markings,mark=at position
			5.5pt with {\arrow[semithick, color=\colline]{open triangle 60 reversed}}},
		decoration={markings,mark=at position
			0.99999 with {\arrow[semithick, color = \colfill]{triangle 60}}},
		decoration={markings,mark=at position
			1 with {\arrow[semithick, color = \colline]{open triangle 60}}},
		double distance=1.4pt,
		shorten <= 5.5pt,
		shorten >= 5.5pt,
		preaction = {decorate},
		postaction = {draw,line width=1.4pt, \colfill,shorten >= 4.5pt, shorten <= 4.5pt},
		color=\colline
		]

		\tikzstyle{RedArrow} = [thick, decoration={markings,mark=at position
			5.5pt with {\arrow[semithick, color=\colfill]{triangle 60 reversed}}},
		decoration={markings,mark=at position
			5.5pt with {\arrow[semithick, color=\colline]{open triangle 60 reversed}}},
		%decoration={markings,mark=at position
		%	0.99999 with {\arrow[semithick, color = \colfill]{triangle 60}}},
		%decoration={markings,mark=at position
		%	1 with {\arrow[semithick, color = \colline]{open triangle 60}}},
		double distance=1.4pt,
		shorten <= 5.5pt,
		%shorten >= 5.5pt,
		preaction = {decorate},
		postaction = {draw,line width=1.4pt, \colfill,shorten >= 4.5pt, shorten <= 4.5pt},
		color=\colline
		]

		\renewcommand{\colline}{light-gray}

		\node[draw,rectangle, color=\colline,align=center,text width = 2cm, minimum height = \minheight, minimum width = 2.4cm] (Pt) at (\boxmid-\boxdist,\y){$\Gamma P_t f = e^{-2Kt} P_t \Gamma f$};
		\node[draw, color=\colline,rectangle,align=center,text width = 2cm, minimum height = \minheight, minimum width = 2.4cm] (phi) at (\boxmid,\y){$f = \varphi + C$, $ -\Delta \varphi = K\varphi$};

		\node[draw, color=\colline,rectangle,align=center,text width = 2cm, minimum height = \minheight, minimum width = 2.4cm] (G2) at (\boxmid+\boxdist,\y){$\Gamma_2f= K\Gamma f$};

		\node[draw, color=\colline,rectangle,align=center,text width = 6.0cm, minimum height = \minheight, minimum width = 2*\boxdist + \boxwidth] (dist) at (\boxmid,0cm+\yi){$f:=d(x_0,\cdot), \quad \color{light-gray} \Deg(x_0)=D$};

		\node[draw, color=\colline,rectangle,align=center,text width = 6.0cm, minimum height = \y-\yi + \minheight + 2*\boxpt, minimum width = 2*\boxdist+\boxwidth+ 4*\boxpt] (distbigbox) at (\boxmid,\y/2 +\yi/2 - \boxpt){};
		
        \renewcommand{\colline}{black}			
		
		\node[draw, color=\colline,rectangle,align=center, text width = 1.2cm, minimum height = \y-(\yi) + \minheight + 2*\boxpt, minimum width = \HSSwidth] (HSS) at (\boxmid + \boxdist+\boxwidth/2+ 2*\boxpt + \HSSwidth/2 + \boxbetween,\y/2 +\yi/2 - \boxpt){$HSS$};

		\renewcommand{\colline}{light-gray}		
		%\node[draw,rectangle,align=center, text width = 1.2cm, minimum height = \y- \yi + \minheight + 2*\boxpt, minimum width = \HSSwidth] (dist) at (\boxmid - \boxdist -\boxwidth/2- 2*\boxpt - \HSSwidth/2  -\boxdist +\boxwidth ,\y/2 + \yi/2 - \boxpt){${\lambda_{\deg_{\max}}}$ $= K$};
		
		\node[draw, color=\colline,rectangle,align=center, text width = 1.2cm, minimum height = \y- \yi + \minheight + 2*\boxpt, minimum width = \HSSwidth] (diam) at (\boxmid - \boxdist -\boxwidth/2- 2*\boxpt - \HSSwidth/2  -\boxdist +\boxwidth ,\yih - \boxpt){$d(x_0,y)$ $= \frac{2D}K$};
		
		%\node[draw,rectangle,align=center, text width = 1.2cm, minimum height = \y- \yi + \minheight + 2*\boxpt, minimum width = \HSSwidth] (diam) at (\boxmid - \boxdist -\boxwidth/2- 2*\boxpt - \HSSwidth/2  -\boxdist +\boxwidth ,\yih - \boxpt){$\|d(x_0,\cdot)\|$ $= \frac{2D}K$};

		%\node[draw,rectangle,align=center,text width = 2.0cm, minimum height = \minheight, minimum width = 18cm] (Pt) at (\boxmid,\y + \minheight){$CD(K,\infty)$};

		\renewcommand{\colline}{black}			
		
		\node[draw, color=\colline,rectangle,align=center,text width = 2.0cm, minimum height = \minheight, minimum width = 2*(\withbig)] (kappa) at (\boxmid,\y - 2*\minheight - 2*\boxpt-4*\ybetween){$\kappa=const.$};

		\node[draw, color=\colline,rectangle,align=center,text width = 2.0cm, minimum height = 3*\minheight + 4*\ybetween+4*\boxpt, minimum width = 2*(\withbig)+4*\boxpt] (kappabigbox) at (\boxmid,\y - \minheight - 2*\ybetween -2*\boxpt){};
		
		\renewcommand{\colline}{black}	
		
		\node[draw, color=\colline,rectangle,align=center,text width = 1.2cm, minimum height = 3*\minheight + 4*\ybetween+4*\boxpt, minimum width = \HSSwidth] (lambda) at (\boxmid - 1.5*\boxwidth - 1.5*\HSSwidth - 3*\boxdist + 3*\boxwidth - 4*\boxpt,\y - \minheight - 2*\ybetween -2*\boxpt){${\lambda_{\deg_{\max}}}$ $= K$};

		\node[draw, color=\colline,rectangle,align=center,text width = 5.0cm, minimum height = \minheight, minimum width = \HSSwidth + \boxdist-\boxwidth + 2*(\withbig)+4*\boxpt ] (CD) at (\boxmid- \boxwidth/2 - \boxdist/2 + \boxwidth/2 + 0.5*\boxwidth - 0.5*\HSSwidth,\y + \minheight){$CD(K,\infty)$ and $x_0 \in V$};

		\renewcommand{\colline}{light-gray}
		
		\node[draw, color=\colline,rectangle,align=center,text width = 2.0cm, minimum height = 4*\minheight + 4*\ybetween + 8*\boxpt, minimum width = \HSSwidth + \boxdist-\boxwidth + 2*(\withbig)+8*\boxpt ] (CDbigbox) at (\boxmid- \boxwidth/2 - \boxdist/2 + \boxwidth/2 + 0.5*\boxwidth - 0.5*\HSSwidth,\y - \minheight/2 - 2*\ybetween - 2*\boxpt){};

		\node[draw, color=\colline,rectangle,align=center,text width = 2.0cm, minimum height = 4*\minheight + 4*\ybetween + 8*\boxpt, minimum width = \boxwidth ] (HC) at (\boxmid + 2*\boxwidth +3*\boxdist -3*\boxwidth + \HSSwidth + 6*\boxpt  ,\y - \minheight/2 - 2*\ybetween - 2*\boxpt){$\frac{2D}{K}$-dim. Hypercube with $\kappa = \frac K 2$};

		\renewcommand{\colline}{light-gray}			
		
		\draw[doubleRedArrow] (Pt.east)  to  node[above] {}  (phi.west); 
		\draw[doubleRedArrow] (G2.west)  to  node[above] {}  (phi.east);		
		
		\renewcommand{\colline}{light-gray}
		
		\draw[doubleRedArrow] (diam.east)  to node[above] {}  (distbigbox.west); 
		\draw[doubleRedArrow] (HSS.west)  to  node[above] {}  (distbigbox.east);

		\draw[doubleRedArrow] (HC.west)  to node[above] {} 
		(CDbigbox.east) ; 
		
		\renewcommand{\colline}{black}	
		\draw[RedArrow] (kappabigbox.west)  to  node[above] {}  (lambda.east);

		\end{tikzpicture}
	}
	\caption{A scheme of Theorem~\ref{thm:Obata INTRO}}
	\label{fig:Obata}
\end{figure}
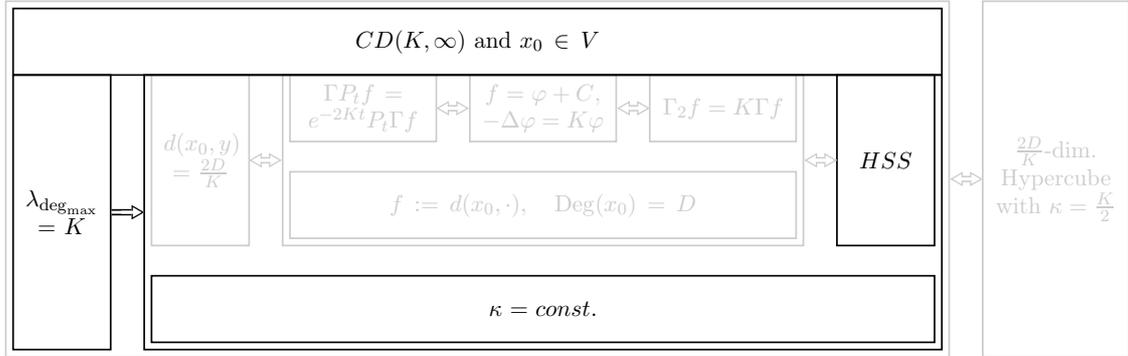

\begin{theorem}[Eigenvalue sharpness]\label{thm:Obata INTRO}
	Let $G=(V,w,m)$ be a connected graph
	with $\D < \infty$ and
	satisfying $CD(K,\infty)$ for some $K>0$.
	%Let $x \in V$ and suppose $\lambda_{\Deg(x)} = K$.
	%Let $x_0$ s.t. $D:=\Deg(x_0)=\D$. and suppose
	%Let $x_0 \in V$.	
	Suppose $K=\lambda_{\deg_{\max}}.$
	%	\textcolor{blue}{
	Then, the following hold true.
		\begin{enumerate}
			\item [1)] 	$G$ satisfies $HSS(\frac{2D}K,\frac K 2,x_0)$ for arbitrary $x_0 \in V$.
			%\item $s_n := \# S_n(x_0) = {D \choose n} $
			{\item [2)] $G$ has constant edge degree.}
		\end{enumerate}
\end{theorem}
This theorem will reappear as Theorem~\ref{thm:Obata}.
%In particular, statement $1)$ is precisely the third of the equivalent statements of Theorem~\ref{thm:abstract CD sharpness INTRO} for $f=d(x_0,\cdot)$.
%Accordingly, the idea to characterize diameter and eigenvalue sharpness is to apply Theorem~\ref{thm:abstract CD sharpness INTRO} to the distance function $f_0=d(x_0,\cdot)$.
In our view, the main achievement in this article is to prove the graph to be a hypercube assuming $CD(K,\infty)$, the hypercube shell structure and a constant edge degree.
%This problem turns out to be purely combinatorial.

\subsection{Small sphere property and non-clustering property}

One key in our approach is to reduce Bakry-\'Emery's curvature-dimension
condition to the combinatorial properties given in
	Definition \ref{def: SSP NCP} below. We remind that
    $d_-^x$ denotes the backwards-degree w.r.t $x$. For unweighted graphs, $d_-^x(y)$
	is the number of neighbors of $y$ closer to $x$ than $y$ itself.

\begin{defn}\label{def: SSP NCP}
	Let $G=(V,E)$ be an unweighted $D$-regular graph, let $K>0$ and let $x \in V$.
	\begin{enumerate}
		\item[(SSP)] We say $x$ satisfies the \emph{small sphere property}
		(SSP) if
		\begin{align*}
		\#S_2(x) \leq {D \choose 2}.
		\end{align*}
		\item[(NCP)] We say $x$ satisfies the \emph{non-clustering property}
		(NCP) if, whenever $d_-^x(z) = 2$ holds for all $z \in S_2(x)$,
		one has that for all $y_1,y_2\in S_1(x)$ there is at most one
		$z \in S_2(x)$ satisfying $y_1\sim z \sim y_2$.
	\end{enumerate}
	We say, $G$ satisfies (SSP) or (NCP), respectively, if (SSP) or
	(NCP), respectively, are satisfied for all $x \in V$.
\end{defn}

We will show that both properties (SSP) and (NCP) follow from
the curvature-dimension condition $CD(2,\infty)$.
Remark that unweighted hypercubes satisfy $CD(2,\infty)$, and therefore as well (SSP) and (NCP). %, \textcolor{red}{defined above in Definition \ref{def:BE-curv}}.

\begin{theorem}[Bakry-\'Emery-curvature, (SSP) and (NCP)]
	\label{thm:local combinatorics CD}
	Let G=(V,E) be a $D$-regular bipartite graph satisfying
	$CD(2,\infty)$ at some point $x \in V$. Then $x$ satisfies the
	small two-sphere property (SSP) and the non-clustering property
	(NCP).
\end{theorem}

This theorem reappears as Theorem~\ref{thm:local_isomorphic}.
We point out the subtlety of (SSP) and (NCP) since already small changes of (NCP) affect our approach that it no longer works (see Lemma~\ref{l:figure no stronger NCP} and Figure~\ref{fig:nostrongerNCP} below).
However, appropriate use of the properties (SSL) and (NCP) defined above allows us to reduce diameter sharpness and eigenvalue sharpness to a purely combinatorial problem which can be solved by a tricky, but direct calculation as stated in the following theorem which will reappear as %Theorem~\ref{thm:local combinatorial characterization} and
 Corollary~\ref{cor:SSL NCP hypercube}.

\begin{theorem}\label{thm:SSL NCP hypercube INTRO}
	Let $G=(V,E)$ be a  graph with the hypercube shell structure $HSS(D,1)$. Suppose, $G$ satisfies (SSL) and (NCP). Then, $G$ is isomorphic to the $D$-dimensional hypercube.
\end{theorem}

%We now turn to a summary of our results. 
%In particular, we now have all ingredients to state the main theorem, i.e., the characterization of the hypercube via Bakry \'Emery curvature. Indeed, assuming the previous results of this section is sufficient to give a short proof of our main theorem.

%\newpage

\subsection{Hypercube characterization}\label{sec: Hypercube char}

Using the concepts explained above, we now characterize the hypercube in the weighted setting.

\begin{theorem}[Main theorem] \label{thm:main} Let $G=(V,w,m)$ be a
	weighted (i.e., without loops and multiple edges)
	connected graph. Let $K>0$. Let $x_0$ s.t. $D:=\Deg(x_0)=\D$.  Let $\deg_{max}$ be the maximal combinatorial degree, i.e. the maximal number of neighbors of a vertex and let
	$0=\lambda_0<\lambda_1 \leq \lambda_2 \leq\ldots$ be the eigenvalues
	of the graph Laplacian $-\Delta$, defined in
	\eqref{eq:lapunweigh} above. The following are
		equivalent:
	\begin{enumerate}
		\item $G$ is a $\frac{2D}{K}$-dimensional hypercube with constant edge degree $\kappa=\frac K 2$.  \label{char:hypercube}
		\item $G$ satisfies $CD(K,\infty)$ and
		$\lambda_{\deg_{\max}} = {K}$. \label{char:Lichnerowicz}
		\item $G$ satisfies $\kappa=const.$ and $CD(K,\infty)$, and
		$\diam_d(G) =  \frac{2D}K$. \label{char:diameter bound}
		\item $G$ satisfies $\kappa=const.$, the hypercube shell structure $HSS\left(\frac{2D}K,\frac K 2\right)$ and $CD(K,\infty)$.  \label{char:CD}
		\item $G$ has constant edge degree $\kappa=\frac K 2$ and the unweighted representation $\widetilde G$ has the hypercube shell structure $HSS\left(\frac{2D}K, 1\right)$ and satisfies (SSP) and (NCP). \label{char:SSP NCP}
	\end{enumerate}
\end{theorem}

A diagram of the proof is given in Figure~\ref{fig:main proof}.
We prove the main theorem under assumption of correctness of all previous results of this section. The correctness of the previous results is shown in the subsequent sections independently of the main theorem.

%[proof of the main theorem]
%We only sketch here the steps of the proof for the reader's convenience:

%	\begin{center}	
		\begin{figure}[h]
			%	\centering	
			\resizebox{15cm}{!}{
				\begin{tikzpicture}[thick]
				
				\node[draw,rectangle] (diam) at (2.8,2.4){$\diam_d=\frac{2D}K$};
				
				\node[draw,rectangle] (HSS_K/2) at (2.8,0){$HSS\left(\frac{2D}{K},\frac K 2\right)$};
				
				\node[draw,rectangle, align=center, text width = 2.2cm] (HSS_1) at (2.8,-3.6){$\widetilde G$ satisfies $HSS\left(\frac{2D}{K},1\right)$};
				
				\node[draw,rectangle] (CD_K_infty) at (6.8,1.2){$CD(K,\infty)$};
				
				\node[draw,rectangle] (lambda=K) at (5.8,-1.2){$\lambda_{\deg_{\max}}=K$};
				
				\node[draw,rectangle] (kappa_const) at (8.5,0){$\kappa=const.$};
				
				\node[draw,rectangle] (kappa=K/2) at (10,-2.4){$\kappa=\frac K 2$};
				
				\node[draw,rectangle, align=center, text width = 1.9cm] (CD_2_infty) at (12,1.2){$\widetilde G$ satisfies  $CD(2,\infty)$};
				
				\node[draw,rectangle, align=center, text width = 2.7cm] (regular bipartite) at (14.5,-3.6){$\widetilde G$ is $\frac{2D}K$-regular $+$ bipartite};
				
				\node[draw,rectangle, align=center, text width = 2.7cm] (SSP NCP) at (17.4,1.2){$\widetilde G$ satisfies  (SSP) + (NCP)};
				
				\node[draw,rectangle,align=center, text width = 2.7cm] (Hypercube) at (19.4,-4.8){$\frac{2D}K$-dimensional hypercube};

				\coordinate (dashed cross CD K infty and diam) at (6.8,2.4){};
				
				\coordinate(dashed cross kappa=K/2 and lambda=K) at (11.3,-1.2){};

				\color{black}
				\draw[dotted,vecArrow] (Hypercube)  |-  (diam);  	
				\draw[dotted,vecArrow] (Hypercube) |- (lambda=K);  
				\draw[dotted,vecArrow] (dashed cross CD K infty and diam) to (CD_K_infty);	
				\draw[dotted,doubleLine] (kappa=K/2) -| (dashed cross kappa=K/2 and lambda=K);	
				\draw[innerWhite] (Hypercube) |- (diam);  	
				\draw[innerWhite] (Hypercube) |- (lambda=K);
				\draw[innerWhite] (dashed cross CD K infty and diam) to (CD_K_infty);
				\draw[innerWhite] (kappa=K/2) -| (dashed cross kappa=K/2 and lambda=K);

				%\color{blue}
				\draw[vecArrow] (lambda=K) |- (HSS_K/2);
				\draw[vecArrow] (lambda=K) |- (kappa_const);  
				\draw[vecArrow] (CD_K_infty) |- (HSS_K/2)
				node[above,pos=0.75]{Theorem~\ref{thm:Obata INTRO}};
				\draw[vecArrow] (CD_K_infty) |- (kappa_const);  
				\draw[innerWhite] (lambda=K) |- (HSS_K/2);
				\draw[innerWhite] (lambda=K) |- (kappa_const);  
				\draw[innerWhite] (CD_K_infty) |- (HSS_K/2);
				\draw[innerWhite] (CD_K_infty) |- (kappa_const); 
				
				%\color{gray}
				\draw[vecArrow] (diam) to (HSS_K/2);  
				\draw[vecArrow] (CD_K_infty) -| (HSS_K/2)
				node[above,pos=0.25]{Theorem~\ref{thm:weighted Cheng INTRO}};
				\draw[innerWhite] (diam) to (HSS_K/2);  
				\draw[innerWhite] (CD_K_infty) -| (HSS_K/2); 
				
				%\color{magenta}
				\draw[vecArrow] (CD_K_infty) --node[above]{Lemma~\ref{l:uniform weights scaling} $(i)$} (CD_2_infty);
				\draw[vecArrow] (kappa=K/2) |- (CD_2_infty); 
				\draw[innerWhite] (CD_K_infty) |- (CD_2_infty);
				\draw[innerWhite] (kappa=K/2) |- (CD_2_infty);

				%\color{red}
				\draw[vecArrow] (HSS_1) --node[above]{Definition~\ref{def:hypercube shell structure}} (regular bipartite);  
				
				%\color{brown}
				\draw[vecArrow] (regular bipartite)|- (SSP NCP);	 
				\draw[vecArrow] (CD_2_infty) -- node[above]{Theorem~\ref{thm:local combinatorics CD}}  (SSP NCP);
				
				\draw[innerWhite] (regular bipartite)|- (SSP NCP);
				
				%\color{orange}
				\draw[vecArrow] (HSS_K/2) to (HSS_1);   
				\draw[vecArrow] (kappa_const) |- (kappa=K/2);
				\draw[vecArrow] (HSS_K/2) |- (kappa=K/2)
				node[above,pos=0.75]{Lemma~\ref{l:uniform weights scaling} $(ii)$} (CD_2_infty);   
				\draw[innerWhite] (HSS_K/2) to (HSS_1);   
				\draw[innerWhite] (kappa_const) |- (kappa=K/2);
				\draw[innerWhite] (HSS_K/2) |- (kappa=K/2);	
				
				%\color{cyan}
				\draw[vecArrow] (HSS_1) |- (Hypercube)
				node[above,pos=0.75]{Corollary~\ref{thm:SSL NCP hypercube INTRO}};	
				\draw[vecArrow] (SSP NCP) |- (Hypercube);
				\draw[innerWhite] (HSS_1) |- (Hypercube);	 
				\draw[innerWhite] (SSP NCP) |- (Hypercube);

				\end{tikzpicture}
			}
			\caption{The figure is a scheme of the proof. The boxes usually stand for properties of $G$. It is mentioned explicitly if they stand for properties of $\widetilde G$. Every arrow has one or more input boxes which represent the assumptions, and output boxes which represent the conclusion of the corresponding theorem. E.g., the dotted arrow has input boxes '$\kappa=\frac K 2$' and '$\frac{2D}K$-dimensional hypercube', and output boxes '$\lambda_{\deg_{\max}}=K$' and '$CD(K,\infty)$' and '$\diam_d=\frac{2D}K$'.}
			\label{fig:main proof}
		\end{figure}
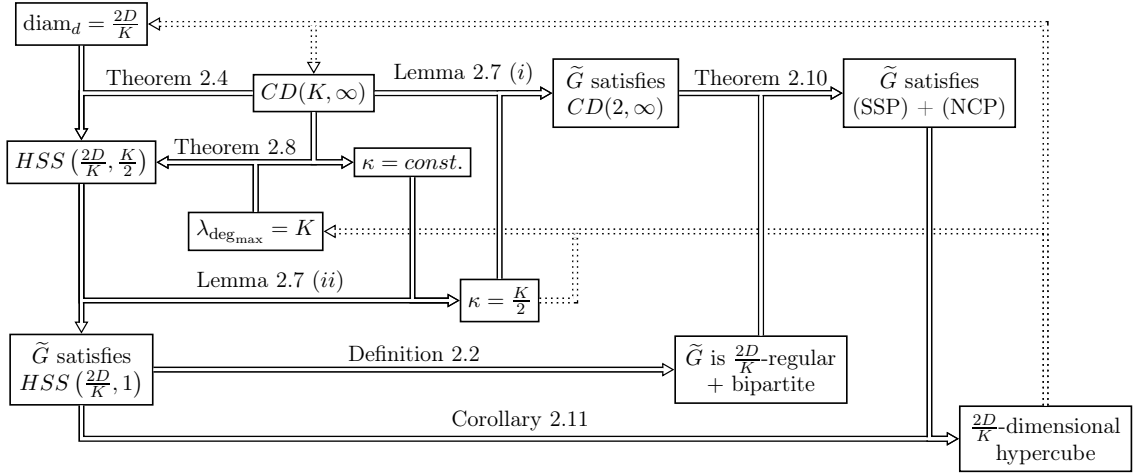			
%	\end{center}
\begin{proof}[Proof of the main theorem]	
	We first notice that the unweighted $\frac {2D}K$-dimensional hypercube satisfies $CD(2,\infty)$, see  \cite{cushing2016bakry,klartag2015discrete,schmuckenschlager1998curvature}.
	By Lemma~\ref{l:uniform weights scaling}(i), we obtain that the $\frac {2D}K$-dimensional hypercube with constant edge degree $\kappa=\frac K 2$ %(visualized by '$\kappa=const.$' in the diagram) 
	satisfies $CD(K,\infty)$.

	The implication \ref{char:hypercube} $\Rightarrow$
	\ref{char:Lichnerowicz} follows since the unweighted 
	hypercube satisfies $\lambda_{\deg_{\max}}=2$ and thus, for the hypercube with constant edge degree $\frac K 2$, we have $\lambda_{\deg_{\max}}=K$.
	
	Similarly, \ref{char:hypercube}
	$\Rightarrow$
	\ref{char:diameter bound} follows since the $\frac{2D}K$-dimensional
	hypercube has diameter $\diam_d(G)=\frac{2D}K$.
	
	%\textcolor{blue}
	{These implications are visualized by the %\textcolor{black} 
		{dotted arrows} 
		in Figure~\ref{fig:main proof}}.
	
	All other theorems, lemmata, corollaries and definitions we refer to in this proof are also shown in Figure~\ref{fig:main proof}.

	The implication \ref{char:Lichnerowicz} $\Rightarrow$
	\ref{char:CD}  follows from 
	Theorem~\ref{thm:Obata INTRO} which is proven via spectral analytic methods, and the implication  \ref{char:diameter bound} $\Rightarrow$ \ref{char:CD} follows from
	Theorem~\ref{thm:weighted Cheng INTRO} which is proven via semigroup properties.

%	We will prove the implication \ref{char:Lichnerowicz} $\Rightarrow$
%	\ref{char:CD}  with Theorem~\ref{thm:Obata INTRO} %\textcolor{blue}
%	using spectral analytic methods, and the implication  \ref{char:diameter bound} $\Rightarrow$ \ref{char:CD} with
%	Theorem~\ref{thm:weighted Cheng INTRO} via semigroup properties.

	The implication \ref{char:CD} $\Rightarrow$ \ref{char:SSP NCP}  
	holds true since Lemma~\ref{l:uniform weights scaling} implies that $\widetilde G$ satisfies $HSS\left(\frac {2D}K,1 \right)$, and that $\kappa=K/2$ and that $\widetilde{G}$ satisfies $CD(2,\infty)$, and therefore, by Definition~\ref{def:hypercube shell structure}  and  Theorem~\ref{thm:local combinatorics CD}, we obtain that $\widetilde G$ satisfies the small sphere property (SSP) and the non-clustering property (NCP).

	The implication \ref{char:SSP NCP} 
	$\Rightarrow$ \ref{char:hypercube} holds true since  Corollary~\ref{thm:SSL NCP hypercube INTRO}  yields that $\widetilde G$, and thus $G$, are  $\frac{2D}K$-dimensional hypercubes.

	Putting together these implications yields the claim of the main
	theorem.
\end{proof}

\begin{figure}[h]
	%	\centering	
	\resizebox{15cm}{!}
	{
		\begin{tikzpicture}[thick]
		
		\tikzstyle{doubleRedArrow} = [thick, decoration={markings,mark=at position
			5.5pt with {\arrow[semithick, color=\colfill]{triangle 60 reversed}}},
		decoration={markings,mark=at position
			5.5pt with {\arrow[semithick, color=\colline]{open triangle 60 reversed}}},
		decoration={markings,mark=at position
			0.99999 with {\arrow[semithick, color = \colfill]{triangle 60}}},
		decoration={markings,mark=at position
			1 with {\arrow[semithick, color = \colline]{open triangle 60}}},
		double distance=1.4pt,
		shorten <= 5.5pt,
		shorten >= 5.5pt,
		preaction = {decorate},
		postaction = {draw,line width=1.4pt, \colfill,shorten >= 4.5pt, shorten <= 4.5pt},
		color=\colline
		]

		\node[draw,rectangle,align=center,text width = 2cm, minimum height = \minheight, minimum width = 2.4cm] (Pt) at (\boxmid-\boxdist,\y){$\Gamma P_t f = e^{-2Kt} P_t \Gamma f$};
		\node[draw,rectangle,align=center,text width = 2cm, minimum height = \minheight, minimum width = 2.4cm] (phi) at (\boxmid,\y){$f = \varphi + C$, $ -\Delta \varphi = K\varphi$};
		\node[draw,rectangle,align=center,text width = 2cm, minimum height = \minheight, minimum width = 2.4cm] (G2) at (\boxmid+\boxdist,\y){$\Gamma_2f= K\Gamma f$};
		\node[draw,rectangle,align=center,text width = 6.0cm, minimum height = \minheight, minimum width = 2*\boxdist + \boxwidth] (dist) at (\boxmid,0cm+\yi){$f:=d(x_0,\cdot), \quad \Deg(x_0)=D$};
		\node[draw,rectangle,align=center,text width = 6.0cm, minimum height = \y-\yi + \minheight + 2*\boxpt, minimum width = 2*\boxdist+\boxwidth+ 4*\boxpt] (distbigbox) at (\boxmid,\y/2 +\yi/2 - \boxpt){};
		\node[draw,rectangle,align=center, text width = 1.2cm, minimum height = \y-(\yi) + \minheight + 2*\boxpt, minimum width = \HSSwidth] (HSS) at (\boxmid + \boxdist+\boxwidth/2+ 2*\boxpt + \HSSwidth/2 + \boxbetween,\y/2 +\yi/2 - \boxpt){$HSS$};
		
		%\node[draw,rectangle,align=center, text width = 1.2cm, minimum height = \y- \yi + \minheight + 2*\boxpt, minimum width = \HSSwidth] (dist) at (\boxmid - \boxdist -\boxwidth/2- 2*\boxpt - \HSSwidth/2  -\boxdist +\boxwidth ,\y/2 + \yi/2 - \boxpt){${\lambda_{\deg_{\max}}}$ $= K$};
		
		\node[draw,rectangle,align=center, text width = 1.2cm, minimum height = \y- \yi + \minheight + 2*\boxpt, minimum width = \HSSwidth] (diam) at (\boxmid - \boxdist -\boxwidth/2- 2*\boxpt - \HSSwidth/2  -\boxdist +\boxwidth ,\yih - \boxpt){$d(x_0,y)$ $= \frac{2D}K$};
		
		%\node[draw,rectangle,align=center, text width = 1.2cm, minimum height = \y- \yi + \minheight + 2*\boxpt, minimum width = \HSSwidth] (diam) at (\boxmid - \boxdist -\boxwidth/2- 2*\boxpt - \HSSwidth/2  -\boxdist +\boxwidth ,\yih - \boxpt){$\|d(x_0,\cdot)\|$ $= \frac{2D}K$};

		%\node[draw,rectangle,align=center,text width = 2.0cm, minimum height = \minheight, minimum width = 18cm] (Pt) at (\boxmid,\y + \minheight){$CD(K,\infty)$};

		\node[draw,rectangle,align=center,text width = 2.0cm, minimum height = \minheight, minimum width = 2*(\withbig)] (kappa) at (\boxmid,\y - 2*\minheight - 2*\boxpt-4*\ybetween){$\kappa=const.$};

		\node[draw,rectangle,align=center,text width = 2.0cm, minimum height = 3*\minheight + 4*\ybetween+4*\boxpt, minimum width = 2*(\withbig)+4*\boxpt] (kappabigbox) at (\boxmid,\y - \minheight - 2*\ybetween -2*\boxpt){};
		
		\node[draw,rectangle,align=center,text width = 1.2cm, minimum height = 3*\minheight + 4*\ybetween+4*\boxpt, minimum width = \HSSwidth] (lambda) at (\boxmid - 1.5*\boxwidth - 1.5*\HSSwidth - 3*\boxdist + 3*\boxwidth - 4*\boxpt,\y - \minheight - 2*\ybetween -2*\boxpt){${\lambda_{\deg_{\max}}}$ $= K$};

		\node[draw,rectangle,align=center,text width = 5.0cm, minimum height = \minheight, minimum width = \HSSwidth + \boxdist-\boxwidth + 2*(\withbig)+4*\boxpt ] (CD) at (\boxmid- \boxwidth/2 - \boxdist/2 + \boxwidth/2 + 0.5*\boxwidth - 0.5*\HSSwidth,\y + \minheight){$CD(K,\infty)$ and $x_0 \in V$};

		\node[draw,rectangle,align=center,text width = 2.0cm, minimum height = 4*\minheight + 4*\ybetween + 8*\boxpt, minimum width = \HSSwidth + \boxdist-\boxwidth + 2*(\withbig)+8*\boxpt ] (CDbigbox) at (\boxmid- \boxwidth/2 - \boxdist/2 + \boxwidth/2 + 0.5*\boxwidth - 0.5*\HSSwidth,\y - \minheight/2 - 2*\ybetween - 2*\boxpt){};

		\node[draw,rectangle,align=center,text width = 2.0cm, minimum height = 4*\minheight + 4*\ybetween + 8*\boxpt, minimum width = \boxwidth ] (HC) at (\boxmid + 2*\boxwidth +3*\boxdist -3*\boxwidth + \HSSwidth + 6*\boxpt  ,\y - \minheight/2 - 2*\ybetween - 2*\boxpt){$\frac{2D}{K}$-dim. Hypercube with $\kappa = \frac K 2$};

		\draw[doubleRedArrow] (HC.west)  to node[above] {6}  (CDbigbox.east) ; 
		\draw[doubleRedArrow] (lambda.east)  to  node[above] {1}  (kappabigbox.west); 
		\draw[doubleRedArrow] (diam.east)  to node[above] {2}  (distbigbox.west); 
		\draw[doubleRedArrow] (HSS.west)  to  node[above] {5}  (distbigbox.east); 
		\draw[doubleRedArrow] (Pt.east)  to  node[above] {3}  (phi.west); 
		\draw[doubleRedArrow] (G2.west)  to  node[above] {4}  (phi.east);

		\end{tikzpicture}
	}
	\caption{The figure is a summary scheme of our results. The five equivalence arrows on the left only hold under the assumption of $CD(K,\infty)$. The box $HSS$ is an abbreviation for the hypercube shell structure $HSS(\frac{2D}K,\frac K 2, x_0)$ introduced in Definition~\ref{def:hypercube shell structure}.
		The edge degree $\kappa$ is defined in Definition~\ref{def:uniform edge degree}.
		%The box $f=\varphi+C$ means	that there exist a constant $C$ and an eigenfunction $\Delta \varphi=-K\varphi$ s.t. $f=\varphi+C$.	
		The leftmost equivalence arrow $\stackrel{1}{\Leftrightarrow}$ (for $\stackrel{1}{\Rightarrow}$, see Theorem~\ref{thm:Obata INTRO}) reads as: 'Assume $CD(K,\infty)$ and $x_0 \in V$. Then $\lambda_{\deg_{\max}}=K$ is equivalent to $\kappa=const$ and $d(x_0,y)=\frac{2D}K$ for some $y$.'
		Both equivalence arrows in the middle, $\stackrel{3}{\Leftrightarrow}$ and $\stackrel{4}{\Leftrightarrow}$ (see Theorem~\ref{thm:abstract CD sharpness INTRO}) , should be interpreted as follows.
		Assuming $CD(K,\infty)$, the equivalence between $\Gamma_2 f = K\Gamma f$ and $f=\varphi + C$ with $-\Delta \varphi = K \varphi$ and $\Gamma P_t f = e^{-2Kt} P_t \Gamma f$ holds for arbitrary $f$. In contrast, the equivalences $\stackrel{2}{\Leftrightarrow}$ and $\stackrel{5}{\Leftrightarrow}$ and (see Theorem~\ref{thm:weighted Cheng INTRO}) of, e.g., $HSS(\frac{2D}K,\frac K 2, x_0)$ and $\Gamma_2 f = K \Gamma f$ only hold for the special choice $f:=d(x_0,\cdot)$. There are subtle methods involved to prove $\stackrel{6}{\Leftrightarrow}$, therefore this equivalence arrow is not covered by a single theorem. 
	}
	\label{fig:results}
\end{figure}
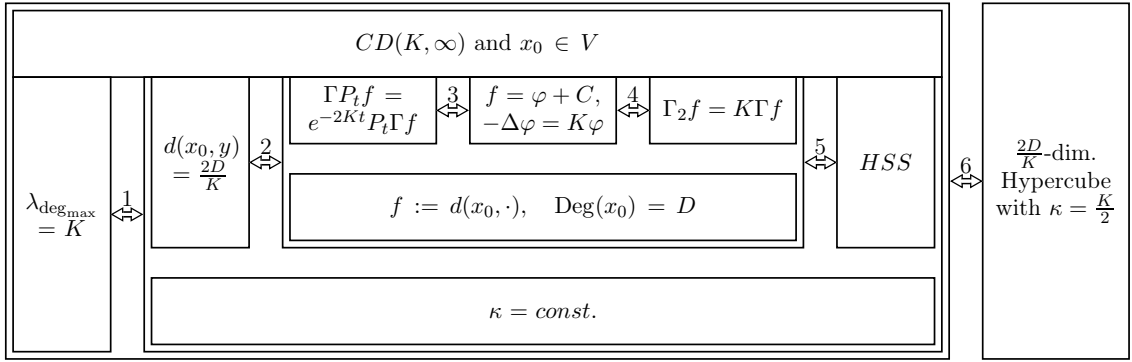

%A summary of our results is given in Figure~\ref{fig:results} below.

%\end{comment}

\section{Sharp curvature dimension inequality}\label{sec:sharp CD}
	This section is dedicated to prove Theorem~\ref{thm:abstract CD sharpness INTRO} which is the abstract characterization of $CD$-sharpness and Lemma~\ref{l:eigenvalue_sharp_CD}  which
	connects eigenvalue sharpness with $CD$ sharpness of the distance function and
	can be seen as the first part towards the proof of a discrete Obata theorem.
	The remaining parts to prove the Obata Theorem are provided in the sections below.
	The classical Obata rigidity theorem states that sharpness of Lichnerowicz eigenvalue bound is only attained for spheres.
	In the discrete setting, we prove that sharpness for the higher order eigenvalue bound is only attained for hypercubes, playing the role of a substitute for the sphere in the manifolds setting. We start giving the discrete Lichnerowicz eigenvalue bound (see \cite[Theorem~2.1]{bauer2014curvature} or  \cite[Theorem~1.6]{liu2015curvature}).

\begin{theorem}[Lichnerowicz eigenvalue bound]\label{thm:Lichnerowicz}
	Let $G=(V,E)$ be a graph satisfying $CD(K,\infty)$ for some $K>0$. Let $0=\lambda_0 \leq \lambda_1 \leq \ldots$ be the eigenvalues of $- \Delta$. Then, $\lambda_1 \geq K$.
\end{theorem}

\begin{example}\label{ex: sharp eigenvalue estimate}
	One is tempted to think that analogously to the Obata Sphere Theorem, sharpness of $\lambda_1 \geq K$ is only attained for hypercubes.
	But this is not true. We have the following counter examples.
	\begin{enumerate}
		\item 	Let $H_D$ the $D$-dimensional hypercube and let $G$ be a graph satisfying $CD(2,\infty)$. Then, the cartesian product $H_D \times G$ satisfies $CD(2,\infty)$ and has first non-zero eigenvalue $\lambda_1=2$.
		\item
		Let $G$ be a square with one diagonal. Then again, $G$ satisfies $CD(2,\infty)$ and has first non-zero eigenvalue $\lambda_1=2$.
	\end{enumerate}
\end{example}

Hence, we need stronger assumptions to characterize the hypercube. The idea in this article is to assume $\lambda_{\D}=K$ instead of the weaker condition $\lambda_1=K$.

\subsection{Geometric properties of eigenfunctions}
The goal of this subsection is to prove Theorem~\ref{thm:abstract CD sharpness INTRO} which is the abstract characterization of $CD$-sharpness.
The crucial step to do so is to show that the distance function to some fixed point, up to some constant, is an eigenfunction  to eigenvalue $K$.

The following lemma is crucial for the proof that
an eigenfunction to the eigenvalue $K$ is already uniquely determined
by its values on a one-ball (see Lemma~\ref{l:unique extension from B1} below).

\newcommand{\phif}{f}

\begin{lemma}\label{l:minimality}
	Let $G=(V,w,m)$ be a weighted
	graph, let $x,z \in X$ with $d(x,z)=2$ and let $ \phif:V \to \IR$.
	Suppose $$\frac{\phif(x) + \phif(z)}2 = \frac{\sum_{y:x\sim y\sim z} \phif(y) w(x,y)w(y,z)/m(y)}{\sum_{y:x\sim y\sim z} w(x,y)w(y,z)/m(y)}.$$
	Then for all $r\neq 0$, we have $\Gamma_2 \phif(x) < \Gamma_2 \left(\phif + r1_{\{z\}}\right)(x)$.
\end{lemma}
%\textcolor{red}{
%Norbert, Shiping, is there a nice way to prove this using your paper with David Cushing?
%For me, it's sort of folklore but it would be rather lengthy to write down if I did.}
%\textcolor{blue}{
\begin{proof} Let $\phif_x\in \mathbb{R}^{\# B_2(x)}$ be the vector given by the restriction of the function $\phif$ on $B_2(x)$. Let $\Gamma_2(x)$ be the $(\# B_2(x)) \times (\# B_2(x))$ symmetric matrix such that $\Gamma_2 \phif(x)=\phif_x^\top \Gamma_2(x)\phif_x$. In fact, the column $\Gamma_2(x)1_{\{z\}}$ of $\Gamma_2(x)$ corresponding to a vertex $z\in S_2(x)$ is given as follows (see \cite[Section 2.3]{cushing2016bakry} and \cite[Section 12]{cushing2016bakry}): 
	 $$ (\Gamma_2(x))_{x,z} = (\Gamma_2(x))_{z,z} = \frac{1}{4 m(x)} \sum_{y: x \sim y \sim z} w(x,y) w(y,z) / m(y) > 0; $$
	 For any $y \in S_1(x), (\Gamma_2(x))_{y,z} = - \frac{1}{2 m(x)} w(x,y)w(y,z)/m(y)$ if $y \sim z$ and $0$ otherwise; Finally,
	 $(\Gamma_2(x))_{z',z}=0$ for any $z' \in S_2(x)$ different from $z$.
	 Therefore, we have
%	$(\Gamma_2(x))_{x,z}=\frac 1 4\sum_{y:x\sim y\sim z} w(x,y)w(y,z)/m(y)/m(x)$; For any $y\in S_1(x)$, $(\Gamma_2(x))_{y,z}=-\frac 1 2w(x,y)w(y,z)/m(y)/m(x)$ if $y\sim z$ and $0$ if otherwise $y\not\sim z$; $(\Gamma_2(x))_{z,z}=\frac 1 4\sum_{y:x\sim y\sim z} w(x,y)w(y,z)/m(y)/m(x)>0$ while $(\Gamma_2(x))_{z',z}=0$ for any $z'\in S_2(x)$ different from $z$. Therefore, we have
	$$\Gamma_2 \left(\phif + r1_{\{z\}}\right)(x)=\phif_x^\top \Gamma_2(x)\phif_x+2\phif_x^\top \Gamma_2(x)1_{\{z\}}+r^2(\Gamma_2(x))_{z,z}>\phif_x^\top \Gamma_2(x)\phif_x=\Gamma_2 \phif(x),$$
	since \begin{align*}
		\phif_x^\top \Gamma_2(x)1_{\{z\}}&=\frac 1 {4m(x)}(\phif(x)+\phif(z))\sum_{y:x\sim y\sim z} w(x,y)w(y,z)/m(y)
		\\&\quad -\frac 1 {2m(x)}\sum_{y:x\sim y\sim z}\phif(y)w(x,y)w(y,z)/m(y)\\&=0
	\end{align*} by assumption. This finishes the proof.
\end{proof}
%}

We denote the heat semigroup operator by
	$P_t = e^{t\Delta}$ (for details see, e.g., \cite{lin2015equivalent,liu2014eigenvalue} and
prove Theorem~\ref{thm:abstract CD sharpness INTRO} reappearing as the following theorem.

\begin{theorem}[Abstract $CD$-sharpness properties]\label{thm:abstract CD sharpness}
	Let $G=(V,w,m)$ be a connected graph
	with $\D < \infty$ and 
	satisfying $CD(K,\infty)$.
	%Suppose $\varphi$ to be an eigenfunction to the eigenvalue $K$.
	Let $f\in C(V) $ be a function.
	The following are equivalent.
	\begin{enumerate}
	\item $\Gamma P_t f = e^{-2Kt} P_t \Gamma f$.
	\item $\Gamma_2 f = K \Gamma f$.
	\item $f= \varphi +C$ for a constant $C$ and an eigenfunction $\varphi$ to the eigenvalue $K$ of $-\Delta$.
	\end{enumerate}
	If one of the above statements holds true, we moreover have
	\begin{enumerate}[(a)]
		\item $\Gamma f = const.$
		\item For all $x,z$ with $d(x,z)=2$, we have
		\begin{equation}\label{eq:eigenfct}
			\frac{f(z) + f(x)}2  = \frac{\sum_{y:x\sim y\sim z} f(y){w(x,y)w(y,z)/m(y)}}{{\sum_{y:x\sim y\sim z} w(x,y)w(y,z)/m(y)}}.
		\end{equation}
	\end{enumerate}
\end{theorem}

\begin{proof}
	We start proving $(1)\Rightarrow (2)$.
	We set $F(s):=e^{-2Ks}P_s(\Gamma P_{t-s}f_0)(x_0)$. Observe that
		\begin{align*}
		F(0)=\Gamma P_t f_0(x_0)\,\,\text{ and }\,\,F(t)=e^{-2Kt} P_t \Gamma f_0(x_0).  %\label{eq:CD_sharp_x_0}
		\end{align*} 
		We compute
		\begin{align*}
		\frac{d}{ds}F(s)&=e^{-2Ks}\left[-2KP_s(\Gamma P_{t-s}f_0)(x_0)+P_s(\Delta \Gamma P_{t-s}f_0)(x_0)-P_s(2\Gamma(P_{t-s}f_0,\Delta P_{t-s}f_0))(x_0)\right]\\
		&=e^{-2Ks}P_s(2\Gamma_2P_{t-s}f_0-2 K\Gamma P_{t-s}f_0)(x_0).
		\end{align*}
		Due to assertion $1$ of the theorem and due to $CD(K,\infty)$, we obtain
		\begin{align*}
		0=F(t)-F(0)=\int_0^t\frac{d}{ds}F(s)ds=\int_0^t e^{-2Ks}P_s(2\Gamma_2P_{t-s}f_0-2 K\Gamma P_{t-s}f_0)(x_0) ds\geq 0.
		\end{align*}
		Hence,
		$e^{-2Ks}P_s(2\Gamma_2P_{t-s}f_0-2K\Gamma
		P_{t-s}f_0)(x_0)=0$ for all $s\in [0,t]$. In particular, this tells us that
		$P_t(2\Gamma_2 f_0-2K\Gamma f_0)(x_0)=0$.  
        Since $\Gamma_2f_0-2K\Gamma f_0 \ge 0$, we conclude
		that $\Gamma_2 f_0 \equiv K\Gamma f_0$ which proves assertion $(2)$.
	
	We prove $(2) \Rightarrow (3)$.
	Integrating yields
	\begin{align}\label{eq:sharp Q}
	-K\left\langle f_0, \Delta f_0 \right\rangle=K \left\langle \Gamma f_0,1 \right\rangle  = \left\langle \Gamma_2 f_0,1 \right\rangle =   -\left\langle \Gamma (f_0,\Delta f_0),1 \right\rangle = \left\langle \Delta f_0, \Delta f_0 \right\rangle
	\end{align}
	where
	$\left\langle f, g \right\rangle := \sum_x
	f(x)g(x)m(x)$.
	
	We spectrally decompose $f_0 = \sum \alpha_i\varphi_i$ where
	$\Delta \varphi_i = - \lambda_i \varphi_i$ with
		$\left\langle \varphi_i,\varphi_j \right\rangle = \delta_{ij}$
	and $0=\lambda_0<\lambda_1 \leq \cdots$.
	
	Then,
	$-K\left\langle f_0, \Delta f_0 \right\rangle = K\sum\ \lambda_i
	\alpha_i^2$ and
	$\left\langle \Delta f_0, \Delta f_0 \right\rangle = \sum\
	\lambda_i^2 \alpha_i^2$.
	
	Applying (\ref{eq:sharp Q}) yields
	$$
	0=\sum_i \lambda_i \left[\lambda_i - K \right] \alpha_i^2.
	$$
	
	Lichnerowicz yields $\lambda_1\geq K$ (see
	\cite[Theorem~1.6]{liu2015curvature}) and thus,
	$\lambda_i \left[\lambda_i - K \right] \geq 0$ for all $i\geq0$.
	%We infer $\lambda_1=K$ and $\alpha_i=0$ for $i\geq 2$.
	Therefore, all terms of $\sum_i \lambda_i \left[\lambda_i - K \right] \alpha_i^2$ need to zero which implies $\alpha_i=0$ whenever $\lambda_i \notin\{0,K\}$.
	Thus,
	we can write $f_0 = C + \varphi$ with $\Delta \varphi = -K\varphi$
	and constant $C$.
	
	%We prove $3 \Rightarrow 1.$

	We prove $(3) \Rightarrow (a)$ which will be used later to prove $(3) \Rightarrow (1)$.
	Due to $CD(K,\infty)$, we have
	\begin{align*}
		K\Gamma \varphi \leq  \frac 1 2 \Delta \Gamma \varphi - \Gamma(\varphi,\Delta \varphi) = \frac 1 2 \Delta \Gamma \varphi + K\Gamma \varphi.
	\end{align*}
	Thus, $\Delta \Gamma \varphi \geq 0$ which implies $\Delta \Gamma \varphi = 0$ since $\langle\Delta g,1 \rangle = 0$ for all functions $g:V\to \IR$.
	Since eigenvalue zero has multiplicity one due to connectedness, we see that $\Gamma \varphi = \Gamma f$ must stay constant.
	
	We now prove $(3) \Rightarrow (1)$.
	Since $\varphi$ is an eigenfunction, we have $P_t \varphi = e^{-Kt} \varphi$.
	Since $f$ and $\varphi$ only differ by a constant, we obtain
	\begin{align*}
	\Gamma P_t f = \Gamma P_t \varphi = \Gamma e^{-Kt} \varphi = e^{-2Kt} \Gamma \varphi. 
	\end{align*}
    We proved already $(3) \Rightarrow (a)$ which means that $\Gamma \varphi = const.$ and thus, $\Gamma \varphi = \Gamma f =  P_t \Gamma f$.
    We conclude
    \begin{align*}
    \Gamma P_t f = e^{-2Kt} \Gamma \varphi =  e^{-2Kt} P_t \Gamma f.
    \end{align*}
    
	We finally prove $(2)\Rightarrow (b)$.	
	We start with $\Gamma_2 f = K \Gamma f$. If there were $x,z \in V$ with
	$d(x,z) = 2$ and (\ref{eq:eigenfct}) violated, then we could change $f$ into $g$
	by changing it only in $z$ such that $g$ satisfies  (\ref{eq:eigenfct}) for the pair
	$x,z \in V$. Since $f$ and $g$ agree on $B_1(x)$, we have $\Gamma f(x) =
	\Gamma g(x)$ and $\Gamma_2 g(x) < \Gamma_2 f(x)$ due to Lemma~\ref{l:minimality}. Then we have
	$\Gamma_2 g(x) < \Gamma_2 f(x) = K \Gamma f(x) = K \Gamma g(x)$, violating the assumption that $G$ is $CD(K,\infty)$. 
\end{proof}

The next lemma states that if we know an eigenfunction on a one-ball, we know it everywhere.

\begin{lemma}\label{l:unique extension from B1}
	Let $G=(V,w,m)$ be a connected graph
		with $\D < \infty$ and 
	satisfying $CD(K,\infty)$. Let $x\in V$. Suppose $\varphi_1,\varphi_2$ are eigenfunctions to eigenvalue $K$. Suppose furthermore $\varphi_1|_{B_1(x)}=\varphi_2|_{B_1(x)}$. Then, $\varphi_1 \equiv \varphi_2$.
\end{lemma}
\begin{proof}
	We prove via induction over the spheres. Due to the above Theorem, $\varphi_i(z)$ is uniquely determined for $z \in S_{k+1}(x)$ whenever we know $\varphi_i(y)$ for all $y \in B_k$ with $k \geq 1$. In particular, $\varphi_1(z)=\varphi_2(z)$ for all $z \in S_{k+1}(x)$ if we assume $\varphi_1|_{B_k(x)} = \varphi_2|_{B_k(x)}$.
\end{proof}

The next lemma tells us that due to high multiplicity, for any given function, there exists an eigenfunction to eigenvalue $K$ which coincides with the given function locally.
		We recall that the combinatorial degree of a vertex $x\in V$  is given by $\deg(x) =\#\{y:y\sim x\}$. We write $\deg_{\max}:=\max_{x\in V}\deg(x)$.

	\begin{lemma}\label{l: One-Ball eigenfunction surjective}
		Let $G=(V,E)$ be a graph satisfying $CD(K,\infty)$ for some $K>0$. Let $x \in V$ and suppose $\lambda_{\deg(x)} = K$.
		Let $f:V \to \IR$ be a function with $\Delta f(x) = -Kf(x)$ at point $x$. Then, there exists an eigenfunction $\varphi$ to eigenvalue $K$ s.t. $\varphi|_{B_1(x)} = f|_{B_1(x)}$.
	\end{lemma}
	
	\begin{proof}
		This follows from a dimension argument.
		Let $\Phi:=\{\varphi:\Delta \varphi = - K \varphi\}$ be the eigenspace to the eigenvalue $K$. By assumption, $\dim \Phi \geq \deg(x)$.
		Let $\Phi|_{B_1(x)} := \{\varphi|_{B_1(x)}:\Delta \varphi = - K \varphi\}$ be the eigenspace restricted to $B_1(x)$.
		Due to Lemma~\ref{l:unique extension from B1}, the map
		$\Phi \to \Phi|_{B_1(x)}$ via $\varphi \mapsto \varphi|_{B_1(x)}$ is an injective linear transformation and thus, $\dim \Phi|_{B_1(x)} \geq \dim \Phi$.	
		Moreover, $\Phi|_{B_1(x)}$ is subspace of $\Psi_{B_1(x)}:=\{g: B_1(x)\to \IR: \Delta g(x) = Kg(x)\}$ which has dimension $\#B_1(x)-1 = \deg(x)$.
		We conclude
		$$\deg(x) \leq \dim \Phi \leq \dim \Phi|_{B_1(x)} \leq \dim \Psi|_{B_1(x)} = \deg(x).$$
		In particular, $\dim \Phi =\dim \Psi|_{B_1(x)}$ and hence,  the map $ \Phi \to \Psi_{B_1(x)}$ via $\varphi \mapsto \varphi|_{B_1(x)}$ is surjective since we already know injectivity.
		For given $f$ with $\Delta f(x) =-Kf(x)$, we have that $f|_{B_1(x)} \in \Psi|_{B_1(x)}$. Due to surjectivity discussed before, there is $\varphi \in \Phi$ satisfying $\varphi|_{B_1(x)} = f|_{B_1(x)}$ as desired.
	\end{proof}

We use the above lemma to prove that, assuming high multiplicity of eigenvalue $K$, one can conclude sharpness of the $CD(K,\infty)$ inequality for the distance function.
\begin{lemma}\label{l:eigenvalue_sharp_CD}
     	Let $G=(V,w,m)$ be a connected graph satisfying $CD(K,\infty)$ for some $K>0$.
		%Let $x \in V$ and suppose $\lambda_{\Deg(x)} = K$.
		%Let $x_0$ s.t. $D:=\Deg(x_0)=\D$. and suppose
		Let $x_0 \in V$.	
		Suppose $\lambda_{\deg_{\max}}=K.$
		%	\textcolor{blue}{
		Then, $\Gamma_2 f = K \Gamma f$ with $f= d(x_0,\cdot)$.
		%	}
\end{lemma}

\begin{proof}
		Let $\psi:V\to \IR$ be given by $\psi(y):= d(x_0,y) - D/K$.
		Then, $-\Delta \psi(x_0) = D = K\psi(x_0)$.
		Hence by Lemma~\ref{l: One-Ball eigenfunction surjective}, there is an eigenfunction $\varphi$ to eigenvalue $K$ s.t. $\varphi|_{B_1(x_0)} = \psi|_{B_1(x_0)}$.
		Due to Theorem~\ref{thm:abstract CD sharpness}, we have
		$$ \varphi(z) = - \varphi(x) + 2\frac{\sum_{y:x\sim y\sim z} \varphi(y)w(x,y)w(y,z)/m(y)}{\sum_{y:x\sim y\sim z} w(x,y)w(y,z)/m(y)}.$$
		for all $x,z$ with $d(x,z)=2$.
		Since the same equation holds for $\psi$ whenever $d(z,x_0) = 2+d(x,x_0)$, we conclude $\varphi = \psi$.
		Since $\Gamma_2$ and $\Gamma$ are invariant under adding constants and due to Theorem~\ref{thm:abstract CD sharpness}, this implies $\Gamma_2 d(x_0,\cdot)=\Gamma_2 \psi = K\Gamma \psi = K\Gamma d(x_0,\cdot)$.
		This finishes the proof.
\end{proof}

\subsection{An upper bound for the multiplicity of the eigenvalue $K$}
The methods above have shown that eigenfunctions to the eigenvalue $K$ are already uniquely determined by its values on a one-ball.
We will use a simple dimension argument to obtain an upper bound for the multiplicity of the eigenvalue $K$
\begin{theorem}\label{thm: multiplicity estimate}
 Let $G = (V,w,m)$ be a connected graph with $\D < \infty$ and satisfying $CD(K,\infty)$ for some $K > 0$. Then we have $\lambda_1 \geq K$ and, if $K$ is an eigenvalue of $-\Delta$, then its multiplicity is at most $\min_{x \in V} {\rm{deg}}(x)$. 
\end{theorem}
\begin{proof}
We first observe that $G$ is finite due to the diameter bound (see \cite[Corollary~2.2]{liu2016bakry}).
The inequality $\lambda_1 \geq K$ follows from Lichnerowicz inequality (see \cite[Theorem~2.1]{bauer2014curvature} or  \cite[Theorem~1.6]{liu2015curvature}).

We now prove the upper bound of the multiplicity.
Choose a $1$-ball $B_1(x)$. Let $x \in V$ for which we have ${\rm{deg}}(x) = \min_{y \in V} {\rm{deg}}(y)$. Due to Lemma~\ref{l:unique extension from B1}, the eigenfunctions to the eigenvalue $K$
are uniquely determined by the values on $B_1(x)$. Using the subspace $\Phi \vert_{B_1(x)}$ introduced in the proof of Lemma~\ref{l: One-Ball eigenfunction surjective}, we know its dimension is equal to the multiplicity of the eigenvalue $K$. 
On the other hand, we have $\Phi \vert_{B_1(x)} \subseteq \IR^{B_1(x)}$
and $\Phi \vert_{B_1(x)}$ does not contain any constant vectors. Therefore, this vector space must have dimension at most ${\#B_1(x) - 1 = \rm{deg}}(x)$. This finishes the proof.
\end{proof}

\begin{rem}
We will show in Sections~\ref{sec:combinatorial approach} and \ref{sec: combinatorial Char} that multiplicity equals ${\rm{deg}}_{\max}$  implies that $G$ is the $D$-dimensional hypercube. It is an interesting question whether, for given $1 \le k < {\rm{deg}}_{\max}$, there is also a characterization of all connected graphs with $\D<\infty$ and satisfying $CD(K,\infty)$. 
\end{rem}

\section{Sharp curvature estimates and the distance function}\label{sec:Sharp diameter bounds}
	This section is dedicated to prove both, Theorem~\ref{thm:weighted Cheng INTRO} which can be seen as part of a discrete Cheng theorem, and Theorem~\ref{thm:Obata INTRO} which can be seen as part of a discrete Obata theorem, presenting diameter or eigenvalue conditions which lead to the same shell structure as the hypercube. Moreover, we explain the necessity of the assumption of an constant edge degree for our discrete Cheng theorem in section~\ref{sec:Examples}.
Semigroup methods allow us to investigate the behavior of the distance
function $f_0 = d(x_0,\cdot)$. In particular, we will be able to
recover coarse sphere structures from diameter sharpness,
%sharpness of Bonnet-Myers,
i.e., the size of every sphere and the in- and outgoing degrees of the
vertices. In other words, we will know for every vertex to how many
vertices in the next sphere it is connected, but we do not know to
which ones. 
So in order to establish the full discrete versions of the Cheng and Obata theorems,
we will need further investigations carried out in Sections~\ref{sec:combinatorial approach} and \ref{sec: combinatorial Char} and to prove Theorem~\ref{thm:SSL NCP hypercube INTRO}.
\subsection{Diameter sharpness}
We now study sharpness of the diameter bound obtained in
\cite[Corollary~2.2]{liu2016bakry} via semigroup methods.
The following theorem is the restatement of Theorem~\ref{thm:weighted Cheng INTRO}.

	\begin{theorem}[Diameter sharpness for weighted graphs]\label{thm:weighted Cheng}
		Let $G=(V,w,m)$ be a connected graph satisfying $CD(K,\infty)$ for some $K>0$. Let $x_0 \in V$ and let $f_0 := d(x_0,\cdot)$. Suppose $D:=\D<\infty$. The following are equivalent:
		\begin{enumerate}
			\item There exists $y_0\in V$ s.t. $d(x_0,y_0) = \frac {2D}K$.
			\item $\Deg(x_0)=D$ and $\Gamma P_t f_0 = e^{-2Kt} P_t \Gamma f_0$.
			\item $\Deg(x_0)=D$ and $\Gamma_2 f_0 = K \Gamma f_0$.
			\item $\Deg(x_0)=D$ and $f_0= \varphi +C$ for a constant $C$ and an eigenfunction $\varphi$ to the eigenvalue $K$ of $-\Delta$.
			\item $G$ has the  hypercube shell structure $HSS\left(\frac{2D}K,\frac K 2,x_0 \right)$.
		\end{enumerate}
	\end{theorem}

In Corollary~\ref{cor:HSS but no hypercube}, we will give an example of graphs apart from the hypercube which satisfy the assertions of the theorem.
Before proving the theorem, we construct an example with the hypercube shell structure which does not have any positive curvature bound.
\begin{example}[Hypercube shell structure and non-positive curvature]\label{Ex:HSS no CD}
	The unweighted graph given in Figure~\ref{fig:HSSnegCurv} obviously satisfies $HSS(4,1,x)$.
	However, the punctured two-ball $\mathring{B}_2(x)$ is not connected, and due to \cite[Theorem~6.4]{cushing2016bakry}, this implies that $CD(0,\infty)$ is not satisfied at vertex $x$.

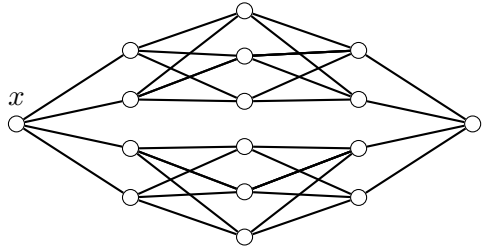
\begin{figure}[h]
	\centering
	\begin{tikzpicture}[scale=0.25]
	\vertex[label=$x$](x0) at (-6,0) {};
	\vertex(y1) at (0,3.9) {};
	\vertex(y2) at (0,1.3) {};
	\vertex(y3) at (0,-1.3) {};
	\vertex(y4) at (0,-3.9) {};
	\vertex(z1) at (6,6) {};
	\vertex(z2) at (6,3.6) {};
	\vertex(z3) at (6,1.2) {};
	\vertex(z4) at (6,-1.2) {};
	\vertex(z5) at (6,-3.6) {};
	\vertex(z6) at (6,-6) {};
	
		\vertex[label=](w0) at (18,0) {};
		\vertex(v1) at (12,3.9) {};
		\vertex(v2) at (12,1.3) {};
		\vertex(v3) at (12,-1.3) {};
		\vertex(v4) at (12,-3.9) {};

	%\tikzset{EdgeStyle/.style={->}}
	\Edge(x0)(y1)
	\Edge(x0)(y2)
	\Edge(x0)(y3)
	\Edge(x0)(y4)
	\Edge(z3)(y1)
	\Edge(z3)(y2)
	\Edge(z5)(y3)	
	\Edge(z6)(y3)	
%	\Edge(z3)(y3)
%	\Edge(z3)(y4)
%	\Edge(z4)(y1)
%	\Edge(z4)(y2)
	\Edge(z4)(y3)
	\Edge(z4)(y4)
	\Edge(z1)(y1)
	\Edge(z2)(y2)
	\Edge(z2)(y2)	
	\Edge(z5)(y3)
	\Edge(z5)(y4)	
	\Edge(z6)(y4)	
	\Edge(z1)(y2)				
	\Edge(z2)(y1)

	\Edge(w0)(v1)
	\Edge(w0)(v2)
	\Edge(w0)(v3)
	\Edge(w0)(v4)
	\Edge(z3)(v1)
	\Edge(z3)(v2)
	\Edge(z5)(v3)	
	\Edge(z6)(v3)	
	%	\Edge(z3)(v3)
	%	\Edge(z3)(v4)
	%	\Edge(z4)(v1)
	%	\Edge(z4)(v2)
	\Edge(z4)(v3)
	\Edge(z4)(v4)
	\Edge(z1)(v1)
	\Edge(z2)(v2)
	\Edge(z2)(v1)	
	\Edge(z5)(v3)
	\Edge(z5)(v3)	
	\Edge(z6)(v4)	
	\Edge(z1)(v2)	
	\Edge(z5)(v4)				
	\Edge(z2)(v1)	
	\Edge(z2)(v1)		
	\end{tikzpicture}
	\caption{The graph satisfies $HSS(4,1,x)$ but it has negative Bakry-\'Emery curvature at vertex $x$.}\label{fig:HSSnegCurv}
\end{figure}
\end{example}
\begin{proof}[Proof of Theorem~\ref{thm:weighted Cheng}]
	
	We recall that the hypercube shell structure $HSS\left(\frac {2D}K,\frac K 2,x_0 \right)$ means
	\begin{enumerate}
		\item [a)] $G$ is $D$-regular w.r.t $\Deg$ defined in (\ref{eq:defDeg}),
		\item [b)] $G$ is bipartite,
		\item [c)] $d_-^{x_0}(x) = \frac K 2 d(x,x_0)$ for all $x \in V$.
	\end{enumerate}

	First, we prove $1 \Rightarrow 2$.
	{We remark that $G$ is finite due to finite (combinatorial)
		diameter  (\cite[Corollary~2.2]{liu2016bakry}), bounded above by $\frac{2D} K$.
		Let $f_0(\cdot):=d(x_0,\cdot): V\to \mathbb{R}$.}
	Similar to the proof of \cite[Theorem~2.1]{liu2016bakry}, we have $|\Delta g| \leq \sqrt{2 \D \Gamma g}$ and therefore,
	\begin{align}
	\frac{2 \D} K &= \diam_d(G)   =f_0(y_0) - f_0(x_0)
	\leq \int_0^\infty |\Delta P_t f_0(x_0)|  + |\Delta P_t f_0(y_0)|dt \nonumber\\
	&\leq \int_0^\infty \sqrt{2 \D \Gamma P_t f_0(x_0)}  + \sqrt{2 \D \Gamma P_t f_0(y_0)} dt \label{eq:Deg sharp}\\
	&\stackrel{CD(K,\infty)}{\leq} \sqrt{2\D} \int_0^\infty e^{-Kt} \left(\sqrt{P_t \Gamma f_0(x_0)} +\sqrt{ P_t \Gamma f_0(y_0)}\right) dt \label{eq:eKT}\\
	&\leq \sqrt{2\D} \int_0^\infty e^{-Kt}\cdot 2\sqrt{\|\Gamma f_0\|_\infty} dt\nonumber\\
	&\leq \sqrt{2\D} \int_0^\infty e^{-Kt}\cdot 2\sqrt{\frac{\D}{2}} dt \nonumber\\
	&=\frac{2 \D} K \nonumber.
	\end{align}
	Hence, we have equality in every step of the calculation.
	Due to sharpness of (\ref{eq:eKT}), we have
	\begin{align*}
	\Gamma P_t f_0(x_0) = e^{-2Kt} P_t \Gamma f_0(x_0) %\label{eq:CD sharp x0}
	\end{align*}
	for all $t \geq 0$. 
	Due to sharpness of (\ref{eq:Deg sharp}), we have $\Deg(x_0)=\D=D$ which proves assertion $2$ of the theorem.
	
	The equivalence of statements $2$, $3$, and $4$ of the theorem follows from Theorem~\ref{thm:abstract CD sharpness}. 
	
	We prove $4 \Rightarrow 5$.
	 We first prove $D$-regularity and bipartiteness.	
	By Theorem~\ref{thm:abstract CD sharpness}(a), we have for all $x \in V$ that $\Gamma f_0 (x) = const. = \Gamma f_0(x_0) = D/2$.
	Hence
	for all $x \in V$, 
		\begin{align}
		D=\D = 2\Gamma f_0(x) &= \sum_{\stackrel{y\sim x}{f_0(y) \neq f_0(x)}} \frac{w(x,y)}{m(x)}  \nonumber\\&= \Deg(x) - \sum_{\stackrel{y\sim x}{f_0(y) = f_0(x)}} \frac{w(x,y)}{m(x)}. \label{eq:D slash D}
		\end{align} 
	Since we always have $\Deg(x) \leq \D$, equation (\ref{eq:D slash
		D}) implies $\Deg(x)=\D$ and there is no $y \sim x$ with
	$f_0(x) = f_0(y)$, i.e. there are no edges within the spheres
	$S_k{(x_0)}$.  This proves $D$-regularity and bipartiteness since bipartiteness is
	equivalent to having no edges within the spheres around a fixed
	vertex.
	
	We calculate how $f_0$ decomposes into an eigenfunction $\varphi$ and a constant $C$.
	We have
	\begin{align}
	D=\Delta f_0(x_0) = \Delta \varphi (x_0) = -K\varphi(x_0).
	\end{align}
	Thus, $C=f_0(x_0) - \varphi(x_0) = D/K$ which implies
	$\Delta f_0 = \Delta \varphi = -K\varphi = D - Kf_0$.

	Due to $D$-regularity and bipartiteness, we have $d_-^{x_0}(x) + d_+^{x_0}(x)=D$ for all $x \in V$. On the other hand since $\Delta f_0 = D - Kf_0$, we obtain
	\begin{align*}
	d_+^{x_0}(x) - d_-^{x_0}(x) = \Delta f_0(x) =  D-Kd(x,x_0).
	\end{align*}
	Subtracting the latter equation from $d_-^{x_0}(x) + d_+^{x_0}(x)=D$ yields $2d_-^{x_0}(x) = Kd(x,x_0)$ which proves c) of the hypercube shell 	   structure and thus assertion 5 of the theorem.
	
	We continue proving $5 \Rightarrow 1$.
	Due to $HSS$, we have
	$d_+^{x_0}(x)=D-d_-^{x_0}(x)= D- \frac K 2 d(x,x_0) >0$ whenever $d(x,x_0)< \frac {2D}K$.
	Hence, there exists $y \in V$ with $d(y,x_0) > d(x,x_0)$ as soon as $d(x,x_0)< \frac {2D}K$. By induction principle there exists $y_0\in V$ s.t. $d(x_0,y_0)=\frac{2D}K$ which proves assertion 1 of the theorem.	
\end{proof}

\subsection{Eigenvalue sharpness}\label{sec:_eigenvalue_sharpness}

In this subsection, we prove Theorem~\ref{thm:Obata} which states that
sharpness of the Lichnerowicz eigenvalue estime or the first $\deg_{\max}$
non-trivial eigenvalues implies the hypercube shell structure and constant
edge degree. For the definition of constant edge degree, see
Definition~\ref{def:uniform edge degree}.

We now restate Theorem~\ref{thm:Obata INTRO} for convenience and
provide the proof.

\begin{theorem}[Eigenvalue sharpness]\label{thm:Obata}
  Let $G=(V,w,m)$  be a  connected graph  with $\Deg_{\max}  < \infty$
  satisfying $CD(K,\infty)$ for some $K>0$.
  % Let $x \in V$ and suppose $\lambda_{\Deg(x)} = K$.
  % Let $x_0$ s.t. $D:=\Deg(x_0)=\D$. and suppose
  Let $x_0 \in V$.  Suppose $\lambda_{\deg_{\max}}=K.$
  %	\textcolor{blue}{
  Then, the following hold true.
  \begin{enumerate}
  \item [1)] $G$ satisfies $HSS(\frac{2D}K,\frac K 2,x_0)$ for
    arbitrary $x_0 \in V$.
    % \item $s_n := \# S_n(x_0) = {D \choose n} $
  \item [2)] $G$ has constant edge degree.
  \end{enumerate}
  % }
\end{theorem}

\begin{proof}
		
  % To prove the hypercube shell structure, we have to show
  % \begin{enumerate}
  % \item [a)] $G$ is $D$-regular,
  % \item [b)] $G$ is bipartite,
  % \item [c)] $d_-^{x_0}(x_0) = \frac K 2 d(x_0,y)$ for all
  %   $x_0,y\in V$.
  % \end{enumerate}
  %		
  % For convenience, we write $D:=\D$.  \textcolor{red}{Let $x_0$
  % s.t. $\Deg(x_0)=D$.}  We start proving a) and b).
		
  We start proving 1). We observe that
  Lemma~\ref{l:eigenvalue_sharp_CD} yields $\Gamma_2 f_0 = K\Gamma f_0$
  with $f=d(x_0,\cdot)$.
  % On the other hand due to Theorem~\ref{thm:abstract CD sharpness},
  % we have
  % $$D \geq \Deg(y) \geq d_-^{x_0}(y) + d_+^{x_0}(y) =2\Gamma f(y) =
  % 2\Gamma f(x_0) = D$$
  % for all $y \in V$.
  % Hence, $G$ is $D$-regular and bipartite which proves a) and b).
  % We continue proving c).
  % \textcolor{red}{We observe that the assumption $\Deg(x_0) =D$
  % indeed is no restriction at all due to $D$-regularity.}
  % Combining $d_-^{x_0}(y) + d_+^{x_0}(y) = D$ and
  % $$d_+^{x_0}(y) - d_-^{x_0}(y) = \Delta f(y) = - K f(y) = D- K
  % d(x_0,y)$$
  % yields $d_-^{x_0}(x_0) = \frac K 2 d(x_0,y)$ which proves
  % c). %\textcolor{red}{If} $G$ has standard weights, we know $d_-^{x_0}(y)=1$ for $y \in S_1(x_0)$. But this implies $K=2$ and proves 3) and 4).
  % Putting together a), b) and c) yields 1).
  Therefore, assertion (3) of Theorem~\ref{thm:weighted Cheng} holds
  true when choosing $x_0$ s.t. $\Deg(x_0)$ is maximal.  Now we apply
  $(3)\Rightarrow(5)$ of Theorem~\ref{thm:weighted Cheng} and conclude
  that $G$ satisfies $HSS(\frac{2D}K,\frac K 2,x_0)$. The hypercube
  shell structure (Definition~\ref{def:hypercube shell structure})
  implies that $G$ has constant vertex degree and, therefore,
  assumption $(3)$ and property $(5)$ of Theorem~\ref{thm:weighted
    Cheng} holds true for choosing $x_0$ arbitrary. This finishes the
  proof of 1).

  Next, we prove 2). Recall from Lemma~\ref{l:uniform edge degree
    basic} that a connected graph $G=(V,w,m)$ has constant edge degree
  $\kappa_0$ iff there exist global $m,w>0$ s.t. $w(x,y) \in \{0,w\}$
  and $m(x)=m$ for all $x,y \in V$ and if $\kappa_0=w/m$.

  We first prove that $m$ is constant. Suppose this is not the case.
  Due to connectedness of $G$, there exist $x\sim y$
  s.t. $m(x)>m(y)$. Let $f:V \to \IR$ be a function s.t.  $f(z)=1$ for
  all $z \neq y$ and s.t. $\Delta f(x) = -K$, that is,
  $f(y) = 1 - Km(x)/w(x,y) \neq 1$. By Lemma~\ref{l:
    One-Ball eigenfunction surjective}, there exists an eigenfunction
  $\varphi$ to the eigenvalue $K$ s.t. $\varphi(z)=f(z)$ for
  $z \in B_1(x)$. Hence,
  \begin{align}
    0< \Gamma \varphi(x)= \frac 1 {2m(x)} w(x,y)(f(x)-f(y))^2.
    \label{eq:Gamma_phi_x_for_const_edge_degree}
  \end{align}			
  By Theorem~\ref{thm:abstract CD sharpness}(a), the gradient
  $\Gamma \varphi$ is constant and by assumption, one has $m(x)>m(y)$,
  and thus,
  \begin{align*}
    \Gamma\varphi(x)=\Gamma \varphi(y) \geq \frac 1 {2m(y)} w(x,y)(f(x)-f(y))^2  > \frac 1 {2m(x)} w(x,y)(f(x)-f(y))^2.
  \end{align*}
  This is a contradiction to
  (\ref{eq:Gamma_phi_x_for_const_edge_degree}) and hence $m$ is
  constant.

  Now suppose $G$ has no constant edge degree. By connectedness of $G$, this
  implies that there exists $x$ and $y_i\sim x$ for $i=1,2$ with
  $w(x,y_1)\neq w(x,y_2)$. We know from assertion $1)$ of the theorem
  that $HSS(\frac{2D}K,\frac K 2, x)$, and in particular, using property (3)
  of the hypercube shell structure (Definition \ref{def:hypercube shell structure})
  \begin{align*}
    w(x,y_i)= d_-^{x}(y_i)m(y_i)=\frac K 2 d(x,y_i)m(y_i)
    =\frac K 2 m(y_i) =\frac K 2 m(x)
  \end{align*}
  where the first and the third equality follow from $x\sim y_i$ for
  $i=1,2$.  Thus, $w(x,y_1)=w(x,y_2)$ which is a contradiction.  We
  conclude that $G$ has constant edge degree.
\end{proof}
	
%	We remark that this theorem immediately proves
% \ref{char:Lichnerowicz} $\Rightarrow$ \ref{char:eigenvalue
% sharpness} of the main theorem (Theorem~\ref{thm:main}).

	%	We remark that this theorem immediately proves \ref{char:Lichnerowicz} $\Rightarrow$ \ref{char:eigenvalue sharpness} of the main theorem (Theorem~\ref{thm:main}).

%\textcolor{blue}{
	\subsection{The necessity of a constant edge degree assumption}\label{sec:Examples}
	For the weighted case, one could hope that, whenever a weighted graph satisfies $CD(K,\infty)$ and $\diam(G)=\frac{2\D}{K}$, the graph has to be a hypercube. But that is not true in general.
	In this subsection, we give counter examples. To do so,
	we give a method to transfer spherically symmetric graphs into linear graphs, i.e., weighted graphs with the adjacency of $\IN$ (see \cite{keller2013volume}).
	This transfer preserves Bakry-\'Emery curvature and therefore, the  linear graph corresponding to the hypercube $H_D$ still satisfies $CD(2,\infty)$ and has diameter $D$.
	Using this method, we show that the main theorem fails without the assumption of constant edge degrees.
	We start giving examples with sharp diameter bounds
	According to \cite{keller2013volume}, we define weak spherical symmetry.
	\begin{defn}
		We call a graph $G=(V,w,m)$ \emph{weakly spherically symmetric} w.r.t. a root $x_0 \in V$ if for all $y,z$ with $d(y,x_0)=d(z,x_0)$ holds
		\begin{align}
		m(y)=m(z), \qquad d_-^{x_0}(y)=d_-^{x_0}(z), \qquad \mbox{ and }  d_+^{x_0}(y)=d_+^{x_0}(z).
		\end{align}
	\end{defn}	
	\begin{defn}
		Let $G=(V,w,m)$ be a graph. 
		%For $A,B\subset V$, we write $w(A,B):=\sum_{x \in A, y \in B} w(x,y)$ and $m(A) := \sum_{x\in A}m(x)$.
		Let $x_0 \in V$ and let $G_P^{x_0}= (V_G^{x_0}, w_G^{x_0},m_G^{x_0})$ be given by
		$V_G^{x_0} := \{0,\ldots,\sup_y d(x_0,y)\}$ and
		\begin{align}
		w_G^{x_0}(i,j) = \begin{cases}
		w(S_i(x_0),S_j(x_0))
		&:|i-j|=1 \\0&:  else.
		\end{cases}
		\end{align}
		and
		\begin{align}
		m_G^{x_0}(i):= m(S_i(x_0)).
		\end{align}
		We define $P_G^{x_0} : C(V_G^{x_0}) \to C(V)$ via $ \left( P_G^{x_0} f \right) (x) := f(d(x,x_0))$ for all $x \in V$.
	\end{defn}
	The following lemma is in the spirit of	\cite[Lemma~3.3]{keller2013volume}.
	\begin{lemma}\label{l:Laplace PG commute}
		Let $G=(V,w,m)$ be a weakly spherically symmetric graph. %Suppose there exists $x_0\in V$ s.t. $d_-^{x_0}, d_+^{x_0}$ and $m$ stay constant on the sphere $S_k(x_0)$ for all $k\geq 0$. 
		Then for all $f \in C(V_G^{x_0})$, we have $P_G^{x_0} \Delta f = \Delta P_G^{x_0} f$.
	\end{lemma}	
	\begin{proof}
		Let $f \in C(V_G^{x_0})$, let $x \in V$ and let $n:=x_G^{x_0} := d(x,x_0) \in V_G^{x_0}$.
		Then since $d_-,d_+$ and $m$ are constant on $S_n(x_0)$, we have 
		\begin{align*}
		&\left(P_G^{x_0} \Delta f \right) (x)\\
		=&\Delta f(x_G^{x_0}) \\ =&\frac{  w(S_n(x_0),S_{n+1}(x_0))(f(n+1)-f(n)) + w(S_n(x_0),S_{n-1}(x_0))(f(n-1)-f(n))}{m(S_n(x_0))}\\
		=&\frac{\#S_n(x_0) d_+^{x_0}(x)(f(n+1)-f(n)) + \#S_n(x_0) d_-^{x_0}(x)(f(n-1)-f(n))}{\#S_n(x_0)} \\
		=&\frac 1 {m(x)}\sum_{y\sim x} w(x,y)\left[P_G^{x_0} f(y) - P_G^{x_0} f(x)\right]\\=&
		\Delta \left( P_G^{x_0} f \right) (x).
		\end{align*}  
		This finishes the proof.
	\end{proof}
	We now show that the map $G \mapsto G_P^{(x_0)}$ is curvature preserving if $G$ is weakly spherically symmetric w.r.t. $x_0$.
	\begin{corollary}\label{cor: spherical symmetric CD}
		Let $G=(V,w,m)$ be a weakly spherically symmetric graph. %Suppose there exists $x_0\in V$ s.t. $d_-^{x_0}, d_+^{x_0}$ and $m$ stay constant on the sphere $S_k(x_0)$ for all $k\geq 0$.
		Suppose $G$ satisfies $CD(K,d)$ for some $K,d$. Then,
		$G_P^{x_0}$ also satisfies $CD(K,d)$.
	\end{corollary}
	\begin{proof}
		Obviously for $f,g \in C(V_G^{x_0})$, we have $$P_G^{x_0}(fg) = \left(P_G^{x_0}f \right) \left(P_G^{x_0}g\right)$$
		Together with Lemma~\ref{l:Laplace PG commute}, we obtain
		\begin{align*}
		P_G^{x_0}(\Delta f)^2 &= (\Delta P_G^{x_0}f)^2,\\
		P_G^{x_0}\Gamma f &= \Gamma P_G^{x_0} f,\\
		P_G^{x_0}\Gamma_2 f &= \Gamma_2 P_G^{x_0} f.
		\end{align*}
		To abuse notation, we write $\Delta^2f:=(\Delta f)^2$.
		Since $G$ satisfies $CD(K,d)$, we have
		\begin{align}
		0 \leq \left(\Gamma_2 - K \Gamma - \frac 1 d \Delta^2\right)(P_G^{x_0} f) = P_G^{x_0}\left(\Gamma_2 - K \Gamma - \frac 1 d \Delta^2\right) f 
		\end{align}
		Since $P_G^{x_0} g$ is positive if and only if $g$ is positive, we obtain
		$$\left(\Gamma_2 - K \Gamma - \frac 1 d \Delta^2\right) f \geq 0$$
		which proves that $G_P^{x_0}$ satisfies $CD(K,d)$.
	\end{proof}
	The following lemma gives an explicit representation of ${(H_D)}_P^{x_0}$.
	\begin{lemma}\label{l:H_D spherically symmetric}
		The hypercube $H_D$ is weakly spherically symmetric w.r.t any $x_0 \in V$ and
		$(H_D)_P:={(H_D)}_P^{x_0} = (\{0,\ldots,D\},w_D,m_D)$ with
		\begin{align}
		w_D(k,k+1) =  {D \choose k} \cdot (D-k), \quad m_D(k) = {D \choose k}.
		\end{align}
	\end{lemma}
	\begin{proof}
		We write $H_D=(V,w,m)$.
		We have $m_D(k)=m(S_k(x_0)) = \# S_k(x_0)= {D \choose k}$ for $k=0,\ldots D$.
		Moreover, for every vertex $x \in S_k$ there are exactly $D-k$ edges between $x$ and $S_{k+1}$. Thus, $d_+(x) = D-k$ and 
		\begin{align}
		w_D(k) = w(S_k(x_0),S_{k+1}(x_0)) = \# S_k d_+(x) = {D \choose k} \cdot (D-k)
		\end{align}
		Moreover for $x \in S_k(x_0)$, we have $m(x)=1$ and $d_+(x)=D-k$ and $d_-(x)=k$ which proves weak spherical symmetry of $H_D$.
	\end{proof}
	Now, we can give examples of graphs with hypercube shell structures which are not hypercubes.
	\begin{corollary}\label{cor:HSS but no hypercube}
		The graph ${(H_D)}_P$ satisfies $CD(2,\infty)$ and $\D=D = \diam_d(x_0)$. 
		Moreover, ${(H_D)}_P^{x_0}$ has the hypercube shell structure $HSS(D,1)$.
	\end{corollary}
	\begin{proof}
		Combining Lemma~\ref{l:H_D spherically symmetric} and Corollary~\ref{cor: spherical symmetric CD} with the fact that $H_D$ satisfies $CD(2,\infty)$ yields that ${(H_D)}_P$ satisfies $CD(2,\infty)$. Obviously, ${(H_D)}_P$ has diameter $D$ since the hypercube $H_D$ has. Theorem~\ref{thm:weighted Cheng} yields that ${(H_D)}_P$ has the hypercube shell structure $HSS(D,1)$.
	\end{proof}
	The corollary implies that property (\ref{char:diameter bound}) in the main theorem (Theorem~\ref{thm:main}) is satisfied for ${(H_D)}_P^{x_0}$ except for the constant edge degree $\kappa$ (see Definition~\ref{def:uniform edge degree}), but ${(H_D)}_P^{x_0}$ is no hypercube for $D>1$.
	I.e., the discrete Cheng theorem (Theorem~\ref{thm:main}) fails if we drop the constant edge degree assumption.
Remark that 	${(H_D)}_P^{x_0}$ corresponds to the discrete Ornstein-Uhlenbeck process up to normalization.
%}

\begin{comment}
\begin{figure}[t]
\centering
\includegraphics[width=0.8\linewidth]{"weighted counterexample to Cheng"}
\caption{The figure states that the graph $G_0$ satisfies $CD(1,\infty)$. I.e., $G_0$ gives an example of a weighted graph satisfying sharpness of Bonnet-Myers and Lichnerowicz, but is not a hypercube.}
\label{fig:weightedcounterexampletoCheng}
\end{figure}
\end{comment}

\section{A combinatorial approach to Bakry-\'Emery curvature}\label{sec:combinatorial approach}
From Theorem~\ref{thm:Obata} and Theorem~\ref{thm:weighted Cheng}, we know about coarse structures of the graph. Unfortunately, our semigroup approach cannot distinguish between vertices within the same sphere due to spherical symmetry of $f_0 =d(x_0,\cdot)$. E.g., our semigroup methods cannot see if we replace two edges $(y_1,z_1)$ and $(y_2,z_2)$ by edges  $(y_1,z_2)$ and $(y_2,z_1)$ for $y_i \in S_k(x_0)$ and $z_i \in S_{k+1}(x_0)$ and $i=1,2$.
To have deeper insight into the edge structure between the spheres, we use combinatorial arguments derived from methods in \cite{cushing2016bakry}.

\subsection{Small sphere property and non-clustering property}

We recall the definition of (SSP) and (NCP).
Let $G=(V,E)$ be a $D$-regular graph and let $x \in V$.
\begin{enumerate}
	\item[(SSP)] We say $x$ satisfies the \emph{small sphere property} (SSP) if
	\begin{align*}
	\#S_2(x) \leq {D \choose 2}
	\end{align*}
	\item[(NCP)]   We say $x$ satisfies the \emph{non-clustering property} (NCP) if, whenever 	$d_-^x(z) = 2$ holds for all $z \in S_2(x)$, one has that
	for all $y_1,y_2\in S_1(x)$ there is at most one $z \in S_2(x)$ satisfying $y_1\sim z \sim y_2$. 
\end{enumerate}
We now show that both properties follow from $CD(2,\infty)$ as announced in Theorem~\ref{thm:local combinatorics CD}.

\begin{theorem}[Restatement of Theorem~\ref{thm:local combinatorics CD}]\label{thm:local_isomorphic}
	Let G=(V,E) be a $D$-regular bipartite graph satisfying $CD(2,\infty)$ at some point $x \in V$. Then,
	$x$ satisfies the small two-sphere property (SSP) and the non-clustering property (NCP).
\end{theorem}

\begin{rem}
	Let $x \in V$. Assume $d_-^x(z) = 2$ for all $z \in S_2(x)$.
	Assume further that $x$ satisfies (NCP) and
	 that there is no edge between any two vertices from $S_2(x)$. Then, we can conclude that $B_2(x)$ is isomorphic to the $2$-ball of any vertex in the $D$-dimensional hypercube. 
\end{rem}

For the proof of the theorem, we use \cite[Theorem 9.1 and Proposition 9.9]{cushing2016bakry}. For convenience, we recall those results in the current setting. Let $G=(V,E)$ be an unweighted $D$-regular graph without triangles and $x\in V$. Let $S_1''(x)$ be the graph with vertex set $\{y_i\sim x, i=1,2,\ldots, D\}$ and an edge between $y_i$ and $y_j$ if and only if there exists $z\in S_2(x)$ such that $y_i\sim z\sim y_j$. We assign the following edge weights $w''(y_i,y_j)$ on the edges of $S_1''(x)$:
$$w''(y_i,y_j)=\sum_{z\in S_2(x)}\frac{w(y_i,z)w(z,y_j)}{d^x_-(z)}.$$

Consider the following Laplacian 
$$\Delta_{S_1''(x)}f(y_i)=\sum_{j\in [D]}w''(y_i,y_j)(f(y_j)-f(y_i)).$$
We refer to their eigenvalues $\lambda$ as solutions of $\Delta_{S_1''(x)}f+\lambda f=0$ and list them with their multiplicity by
$$0=\lambda_0(\Delta_{S_1''(x)})\leq \lambda_1(\Delta_{S_1''(x)})\leq\cdots\leq \lambda_{D-1}(\Delta_{S_1''(x)}).$$
\begin{theorem}[\cite{cushing2016bakry}]\label{thm:CLP2016} Let $G=(V,E)$ be an unweighted $D$-regular graph without triangles, $D\geq 2$. Let $x\in V$ and $\Delta_{S_1''(x)}$ be the Laplacian defined as above. Then we have
	\begin{enumerate}
		\item [1)] The vertex $x$ satisfies $CD(2,\infty)$ if and only if $\lambda_1(\Delta_{S_1''(x)})\geq \frac{D}{2}$.
		\item [2)] $\# S_2(x)\leq (D-1)(D-\lambda_1(\Delta_{S_1''(x)}))$.
	\end{enumerate}
\end{theorem}
\begin{rem} Theorem \ref{thm:CLP2016} follows as a special cases of \cite[Theorem 9.1]{cushing2016bakry} and \cite[Proposition 9.9]{cushing2016bakry} . 
Note first that $D$-regularity and
triangle-freeness implies that every vertex $x$ of $G$ is $S_1$-out
regular (i.e., the out-degrees $d_+^y$ of all $y \sim x$ are the same
and agree with $av_1^+(x)$). In this case, it is stated in [4, Theorem
9.1] that the eigenvalue estimate is equivalent to $\infty$-curvature
sharpness and, via the explicit formula of the curvature function,
equivalent to $CD(2,\infty)$, since $(3+D-av_1^+(x))/2=2$.
%In fact, it is stated in \cite{cushing2016bakry} for any $S_1$-out regular vertex $x$ (i.e., the out-degree $d_+^{y}$ of any $y\sim x$ is the same) in a locally finite graph. In that general setting, the condition $CD(0,\infty)$ in $1)$ of Theorem \ref{thm:CLP2016} is replaced by $\infty$-curvature sharpness of $x$.
\end{rem}

We also need the following lemma.
\begin{lemma}\label{l:linear_algebra_rigidity}
	Let $X=(x_{ij})\in \textrm{Sym}(r,\mathbb{R})$ be an $r \times r$ symmetric real matrix with
	\begin{enumerate}
		\item [1)] $x_{ij}\geq 0$, for any $i\neq j \in [r]$.
		\item [2)] $x_{ii}=-\sum_{j\neq i}x_{ij}$.
	\end{enumerate}
	Assume that its eigenvalues (i.e., solutions of $Xf+\lambda f=0$) can be listed with their multiplicity as 
	$$0=\lambda_0\leq \lambda_1\leq \lambda_2\leq\cdots\leq \lambda_{r-1}.$$
	Then we have
	\begin{equation}
	\lambda_1\leq -\frac{\mathrm{Tr}(X)}{r-1},
	\end{equation}
	where the equality holds if and only if $x_{ij}=-\frac{\mathrm{Tr}(X)}{r(r-1)}$ for any $i\neq j$.
\end{lemma}
\begin{proof}
	W.l.o.g., we assume $-\mathrm{Tr}(X)>0$. Since $-\mathrm{Tr}(X)=\sum_{\ell=0}^{r-1}\lambda_\ell \geq (r-1)\lambda_1$, we have $\lambda_1\leq -\frac{\mathrm{Tr}(X)}{r-1}$. The equality implies that $\lambda_1=\lambda_2=\cdots=\lambda_{r-1}=-\frac{\mathrm{Tr}(X)}{r-1}$. Since the eigenspace to $\lambda_0=0$ is spanned by constant vectors, every $f\in \mathbb{R}^r$ orthogonal to constant vectors is an eigenvector of $X$ to the eigenvalue $-\frac{\mathrm{Tr}(X)}{r-1}>0$. It is sufficient to show for any three distinct $i,k,\ell\in [r]$ that $x_{ik}=x_{i\ell}$. Choose $f=e_k-e_\ell$ which is vertical to constant vectors. Then we have $(Xf)_i=x_{ik}-x_{i\ell}=-\frac{\mathrm{Tr}(X)}{r-1}f_i=0$.
\end{proof}
\begin{proof}[Proof of Theorem \ref{thm:local_isomorphic}] Since $x$ satisfies $CD(2,\infty)$, we obtain $\# S_2(x) \leq {D \choose 2}$ by combining $1)$ and $2)$ of Theorem \ref{thm:CLP2016}. I.e., $x$ satisfies (SSP).
	
	We now prove (NCP). Note that there are $D(D-1)$ edges between $S_1(x)$ and $S_2(x)$. Since $d_-^x(z)=2$ for any $z\in S_2(x)$, we conclude $\# S_2(x) = \frac{D(D-1)}{2}={D \choose 2}$. Observe that $\Delta_{S_1''(x)}=:X=(x_{ij})\in \mathrm{Sym}(D,\mathbb{R})$ with $x_{ij}\geq 0$ for all $i\neq j$ and $x_{ii}=-\sum_{j\neq i}x_{ij}$. Moreover, by the construction of $\Delta_{S_1''(x)}$, we have $$\sum_{i,j\in [D], i<j}x_{ij}=\frac{1}{2}{D \choose 2},$$
	since each edge in $S_1''(x)$ contributes a weight $1/2$ and $S_1''(x)$
has $D \choose 2$ edges in total. 
	Therefore, we have $-\mathrm{Tr}(X)=\sum_{i\in [D]}\sum_{j\neq i} x_{ij}={D \choose 2}$. Applying Lemma \ref{l:linear_algebra_rigidity}, we obtain $\lambda_1(\Delta_{S_1''(x)})\leq \frac{D}{2}$. Furthermore, we have by $1)$ of Theorem \ref{thm:CLP2016} that $\lambda_1(\Delta_{S_1''(x)})\geq \frac{D}{2}$. Hence the equality holds and we have $x_{ij}=\frac{1}{2}$ for any $i\neq j$ by Lemma \ref{l:linear_algebra_rigidity}. That is, for any two vertices $y_i,y_j\in S_1(x)$, there is exact one $z\in S_2(x)$ satisfying $y_i\sim z\sim y_j$. This proves (NCP).
\end{proof}

By Theorem~\ref{thm:local_isomorphic}, we directly obtain   \ref{char:CD} $\Rightarrow$ \ref{char:SSP NCP} from the main theorem (Theorem~\ref{thm:main}).

\subsection{The subtleties}

In the following, we demonstrate that already little changes in (NCP) have the consequence that our method no longer works.

\begin{example}
One might be tempted to replace (NCP) by the stronger (NCP2) stating that whenever $\#S_2(x) = {D \choose 2}$, we obtain that for all $y_1,y_2$ there is at most one $z \in S_2(x)$ s.t. $y_1 \sim z \sim y_2$.
But unfortunately, $CD(2,\infty)$ does not imply (NCP2) as one can see in Figure~\ref{fig:nostrongerNCP} and in the following Lemma~\ref{l:figure no stronger NCP}. 
This demonstrates the subtleties of finding a suitable interface between Bakry-\'Emery-curvature and a combinatorial characterization of the hypercube.
\end{example}
\begin{lemma}\label{l:figure no stronger NCP}
	The unweighted graph given in Figure~\ref{fig:nostrongerNCP} satisfies $CD(2,0)$ at point $x$.
\end{lemma}
\begin{proof}
	Since the vertex $x$ is $S_1$-out regular, that is, each vertex in $S_1(x)$ has the same out-degree, we can apply \cite[Theorem 9.1]{cushing2016bakry}. Observe in this example we have $S_1''(x)$ is the complete graph with $4$ vertices, and $w''(y_i,y_j)=\frac{1}{2}$ for any $y_i,y_j \in S_1(x)$. Therefore, we have $\lambda_1(\Delta_{S_1''(x)})=2=\frac{\mathrm{Deg}(x)}{2}$. By \cite[Theorem 9.1]{cushing2016bakry}, we conclude $x$ satisfies $CD(2,\infty)$.
\end{proof}

%\begin{center}
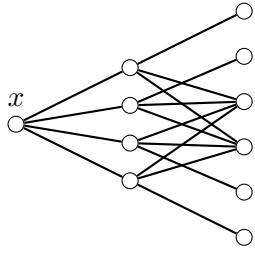
\begin{figure}[h]
	\centering
\begin{tikzpicture}[scale=0.25]
\vertex[label=$x$](x0) at (-6,0) {};
\vertex(y1) at (0,3) {};
\vertex(y2) at (0,1) {};
\vertex(y3) at (0,-1) {};
\vertex(y4) at (0,-3) {};
\vertex(z1) at (6,6) {};
\vertex(z2) at (6,3.6) {};
\vertex(z3) at (6,1.2) {};
\vertex(z4) at (6,-1.2) {};
\vertex(z5) at (6,-3.6) {};
\vertex(z6) at (6,-6) {};
%\tikzset{EdgeStyle/.style={->}}
\Edge(x0)(y1)
\Edge(x0)(y2)
\Edge(x0)(y3)
\Edge(x0)(y4)
\Edge(z3)(y1)
\Edge(z3)(y2)
\Edge(z3)(y3)
\Edge(z3)(y4)
\Edge(z4)(y1)
\Edge(z4)(y2)
\Edge(z4)(y3)
\Edge(z4)(y4)
\Edge(z1)(y1)
\Edge(z2)(y2)
\Edge(z5)(y3)
\Edge(z6)(y4)				
\end{tikzpicture}
\caption{Lemma~\ref{l:figure no stronger NCP} proves that this unweighted graph satisfies $CD(2,\infty)$. Moreover, the graph is bipartite and $B_1(x)$ is $D$-regular with $D=4$. Obviously, $\# S_2(x) = 6 = {D \choose 2}$. I.e., $x$ satisfies all preconditions of (NCP2). But $x$ does not satisfy (NCP2).}
\label{fig:nostrongerNCP}
\end{figure}
%\end{center}

\section{A combinatorial characterization of the hypercube}
\label{sec: combinatorial Char}
The aim of this section is to prove Theorem~\ref{thm:SSL NCP hypercube INTRO} which states that the hypercube shell structure together $HSS(D,1)$ with the small sphere property (SSP) and the non-clustering property (NCP) imply that the graph is a hypercube.
To prove the theorem, we need some preparation.
\subsection{A power set lemma} 

The following lemma will give that every two-sphere $S_2(z)$ around $z \in S_{k+1}(x_0)$ contains at least ${k+1 \choose 2}$ vertices in $S_{k-1}(x_0)$ if we assume that $B_k{(x_0)}$ is isomorphic to a corresponding ball in a hypercube, see (\ref{eq:degRz}).

For sets $X$ and $k \in \IN$, we write
$\iP_k(X) := \{A \subset X: \# A = k\}$ and
$\iP_{\leq k}(X) := \{A \subset X: \# A \leq k\}$.

\begin{lemma}[Power set properties]\label{l:power set}
	 Let $k,D \in \IN$ with $k<D$. Let $A_1,\ldots A_{k+1}$ be pairwise distinct $k$-element subsets of $[D]=\{1,\ldots,D\}$.
	
Then,
\begin{align}\label{eq:Lemma union PkAi}
\# \bigcup_{i=1}^{k+1} \iP_{k-1}(A_i)      \geq {k+1 \choose 2}.
\end{align}
Moreover, equality implies $\# \bigcup_{i=1}^{k+1} A_i = k+1$.	
\end{lemma}

\begin{proof}
We first observe that $\#A_i \cap A_j \leq k-1$ and $\#A_i\cup A_j \geq k+1$ for $i\neq j$.	
We prove for all $j=0,\ldots,k$ that
\begin{align}\label{eq:recursive Union PkAj}
\# \bigcup_{i=1}^{j+1} \iP_{k-1}(A_i) \geq k-j + \# \bigcup_{i=1}^{j} \iP_{k-1}(A_i).
\end{align}

To do so, we calculate
\begin{align}
\# \bigcup_{i=1}^{j+1} \iP_{k-1}(A_i) 
&= \# \bigcup_{i=1}^{j} \iP_{k-1}(A_i) + \#\iP_{k-1}(A_{j+1}) -
\#  \left(  \iP_{k-1}(A_{j+1}) \cap \bigcup_{i=1}^{j} \iP_{k-1}(A_i)   \right)
\nonumber\\&\geq \# \bigcup_{i=1}^{j} \iP_{k-1}(A_i) +k
- \sum_{i=1}^j\#   \left(  \iP_{k-1}(A_{j+1}) \cap  \iP_{k-1}(A_i) \right)
\nonumber\\&\geq k-j + \# \bigcup_{i=1}^{j} \iP_{k-1}(A_i) \label{eq:calc union Pk k-j}
\end{align}
where the last inequality holds due to
$$\#  \left[ \iP_{k-1}(A_i) \cap \iP_{k-1}(A_j)   \right] \leq 1, \quad i\neq j$$
which holds since $B \in \iP_{k-1}(A_i) \cap \iP_{k-1}(A_j)$ implies $B \subset A_i \cap A_j$ and implies $\#B=k-1\geq \#A_i \cap A_j$ and hence, $B = A_i \cap A_j$.

The last calculation implies 
\begin{align}\label{eq:AiAj}
\iP_{k-1}(A_i) \cap \iP_{k-1}(A_j) = \begin{cases}
\iP_{k-1}(A_i)&: i=j\\
A_i \cap A_j&: \# A_i \cup A_j = k+1\\
\emptyset &: else
\end{cases}
\end{align}

Applying (\ref{eq:recursive Union PkAj}) recursively yields

\begin{align*}
\# \bigcup_{i=1}^{k+1} \iP_{k-1}(A_i)     & \geq
\# \bigcup_{i=1}^{k} \iP_{k-1}(A_i)   \\&\geq
 1+ \#\bigcup_{i=1}^{k-1} \iP_{k-1}(A_i) \\& \geq
 2+1 + \#\bigcup_{i=1}^{k-2} \iP_{k-1}(A_i) \\& \vdots \\&\geq
 \sum_{j=0}^{k-1} k-j = {k+1 \choose 2}.
\end{align*}
This proves (\ref{eq:Lemma union PkAi}) and that sharpness implies sharpness of (\ref{eq:recursive Union PkAj}) and (\ref{eq:calc union Pk k-j}) for all $j$. 
We now prove $\# \bigcup_{i=1}^{k+1} A_i = k+1$ in case of sharpness of (\ref{eq:Lemma union PkAi}).
The case $k=1$ is trivially true and we assume $k \ge 2$.
Let $A:= A_1 \cup A_2$. Due to sharpness of (\ref{eq:recursive Union PkAj}) for $j=1$, we have $\#\bigcup_{i=1}^2 \iP_{k-1}(A_i) = 2k-1$ which implies $\#A = k+1$ due to (\ref{eq:AiAj}).

Due to sharpness of (\ref{eq:recursive Union PkAj}) and (\ref{eq:calc union Pk k-j}) for $j=2$,
we have $$\# \left(   \left[ \iP_{k-1}(A_1) \cup \iP_{k-1}(A_2) \right] \cap \iP_{k-1}(A_3) \right) =2$$ which due to (\ref{eq:AiAj}) implies $\# A_1 \cap A_3 = \#A_2 \cap A_3 = k-1$ and $A_1 \cap A_3 \neq A_2 \cap A_3$. Thus, 
$$\#A \cap A_3 = \# \left[  (A_1 \cap A_3) \cup (A_2 \cap A_3) \right]  \geq k = \#A_3$$
 which implies $A_3 \subset A$. Reordering $A_i$ yields $A_i \subset A$ for all $i$. Hence, $\#\bigcup_{i=1}^{k+1} A_i = \# A = k+1$ as desired.
\end{proof}

\subsection{A shell-wise construction of the hypercube}

We recall the symmetric set difference $A \ominus B = (A \cup B) \setminus (A \cap B) $.

Now, we have all ingredients to give a detailed proof of 
Theorem~\ref{thm:SSL NCP hypercube INTRO}.
To do so, we present an even stronger result.
%\ref{char:SSP NCP} $\Rightarrow$ \ref{char:hypercube} of the main theorem (Theorem~\ref{thm:main}).

\begin{theorem}\label{thm:local combinatorial characterization}
	Let $G=(V,E)$ be a $D$-regular bipartite graph and let $k \in \IN$. Suppose there is $x_0 \in V$ s.t. $d_-^{x_0}(y) = d(x_0,y)$ for all $y \in B_k(x_0)$. Suppose the small sphere property (SSP) and the non-clustering property (NCP) (see Definition~\ref{def: SSP NCP}) are satisfied for all $x \in B_{k-2}(x_0)$. Then, $B_k(x_0)$ is isomorphic to the $k$-ball in the $D$-dimensional hypercube.
\end{theorem}

\begin{proof}%[Proof of the theorem]
%$G$ is $D$-regular and bipartite with $D=\D = \diam_d(G)$ and $K=2$ due to assumption.	
	
%We will define an isomorphism $\Phi: V \to \iP([D])$ from the graph $G$ to the hypercube $\iP([D])$ where for $A, B \in \iP([D])$, we have $A \sim B$ if and only if $\# (A \ominus B)=1$.
In the following arguments, we use Definition~\ref{def:hypercube} of the hypercube.
By assumption for $x \in S_j(x_0)$, we have $d_-^{x_0}(x) = j$ and due to bipartiteness and $D$-regularity, $d_+^{x_0}(x)=D-j$ follows immediately.
Hence with using the notation $E(A,B):=\{\{x,y\}\in E: x\in A, y\in B\}$ for $A,B \subseteq V$, we obtain
\begin{align*}
(D-j)\#S_j(x_0) &=\sum_{y\in S_j(x_0)} d_+^{x_0}(y) \\&=\#E(S_j(x_0),S_{j+1}(x_0)) \\&= \sum_{y\in S_{j+1}(x_0)} d_-^{x_0}(y) \\&= (j+1)\#S_{j+1}(x_0). 
\end{align*}
Applying inductively yields
 \begin{align}\label{eq:k sphere size}
\#S_j(x_0) = {D \choose j} 
\end{align}
for all $j\leq k$, assuming $d_-^{x_0}(y) = d(x,y)$ for all $y \in B_k(x_0)$.

Now we prove that we have an isomorphism $\Phi_{\leq k}: B_k(x_0) \cong {\mathcal{P}}_{\le k}([D])$ consistent with adjacency by induction over $k$, which then completes the proof of the theorem. Since $G$ is a $D$-regular graph without triangles, we have an isomorphism $\Phi_{\leq 1}: B_1(x_0) \cong {\mathcal{P}}_{\le 1}([D])$, given by $\Phi_{\le 1}(x_0) = \emptyset$
and $\Phi_{\le 1}{(y_j)} = \{j\}$ for $S_1(x) = \{y_1,\dots,y_D\}$.
This settles the case $k=1$ of the induction.

%Since $G$ is a $D$-regular graph without triangles, we have $B_1(x_0)  \cong \iP_{\leq 1}([D])$. % with $A \sim B$ if and only if $\# (A \ominus B)=1$.	
%We now prove by induction that $B_k(x_0)  \cong \iP_{\leq k}([D])$ for all $k$ which would immediately imply the theorem.

By induction, we assume $B_k(x_0)  \cong \iP_{\leq k}([D])$ via an isomorphism $\Phi_{\leq k} : B_k(x_0)  \to \iP_{\leq k}([D])$ for some $k \geq 1$. We want to show $B_{k+1}(x_0)  \cong \iP_{\leq k+1}([D])$, assuming (SSP) and (NCP) for all $x \in B_{k-1}(x_0)$ and $d_-^{x_0}(y) = d(x_0,y)$ for all $y \in B_{k+1}(x_0)$.

We recall $f_0(x)=d(x,x_0)$ and we define a bipartite graph $(S_{k-1}(x_0) \cup S_{k+1}(x_0), R)$ via 
$(x,y) \in R$ if $f_0(x) \neq f_0(y)$ and if $d(x,y)=2$. 
We write $\deg_R(x) := \#\{y:(x,y) \in R\}$.
The disjoints parts are $S_{k-1}(x_0)$ and $S_{k+1}(x_0)$.

We now show that (SSP) and Lemma~\ref{l:power set} give sharp bounds on $\deg_R$.

For $x \in S_{k-1}(x_0)$, we have by induction assumption, that is, existence of an isomorphism
$\Phi_{\le k}: B_k(x_0) \to {\mathcal{P}}_{\le k}([D])$, that
$$\#  \left[ S_2(x) \cap B_k(x_0) \right] = {D \choose 2} - {D - k + 1 \choose 2}$$ as in the hypercube.
(This identity follows from the fact that, for a given
subset $A \subset [D]$ of cardinality $k-1$, there are precisely
$D-k+1 \choose 2$ subsets $A' \subset [D]$ of cardinality $k+1$ containing $A$).
By (SSP), we have for $x \in S_{k-1}(x_0)$ that
$\# S_2(x) \leq {D \choose 2}$ and thus, 
\begin{align}\label{eq:degRx}
\deg_R(x) =\#  \left[ S_2(x) \cap S_{k+1}(x_0)  \right] \leq {D - k + 1 \choose 2}, \quad \mbox{all } x \in S_{k-1}(x_0).
\end{align}

On the other hand, for all $z \in S_{k+1}(x_0)$, we have by assumption that $d_-^{x_0}(z) = k+1$, say $z \sim y_i$ for $i = 1,\ldots,k+1$ with $y_i \in S_k(x_0)$ pairwise distinct. Due to induction assumption, $y_i$ can be identified with pairwise distinct $A_i := \Phi_{\leq k}(y_i) \in \iP_k([D])$. 
For $x\in B_k(x_0)$, we have $\Phi_{\leq k}(x) \in \bigcup_{i=1}^{k+1}\iP_{k-1}(A_i)$ if and only if $x \in \#S_2(z) \cap S_{k-1}(x_0)$.
Applying Lemma~\ref{l:power set} yields 
\begin{align}\label{eq:degRz}
\deg_R(z)=\#  \left[ S_2(z) \cap S_{k-1}(x_0)  \right]  = \# \bigcup_{i=1}^{k+1} \iP_{k-1}(A_i)  \geq  {k+1 \choose 2}, \quad \mbox{all } z \in S_{k+1}(x_0).
\end{align}

Due to (\ref{eq:k sphere size}), we have $\#S_k(x_0)= {D \choose k}$ and together with (\ref{eq:degRx}) and (\ref{eq:degRz}), we obtain
\begin{align}
\frac {D!}{2(k-1)!(D-k-1)!}
&= {D \choose k+1} {k+1 \choose 2}\label{eq:sharp degR}\\
&= \# S_{k+1}(x_0) {k+1 \choose 2}\nonumber\\
&\stackrel{(\ref{eq:degRz})}{\leq} \sum_{z\in S_{k+1}(x_0)} \deg_R(z)\nonumber\\ &= \#R(S_{k-1}(x),S_{k+1}(x))\nonumber\\
& = \sum_{x\in S_{k-1}(x_0)} \deg_R(x)\nonumber\\
&\stackrel{(\ref{eq:degRx})}{\leq} \# S_{k-1}(x_0)  {D - k + 1 \choose 2} \nonumber\\
&={D \choose k-1}{D-k+1 \choose 2}\nonumber\\
%=\frac{D! }{(k-1)!(D-k+1)!} \cdot \frac {(D-k+1)!}{2(D-k-1)!}\\
&=\frac {D!}{2(k-1)!(D-k-1)!}\nonumber
\end{align}

Thus, we have sharpness and this implies $\deg_R(z) = {k+1 \choose 2}$ for all $z \in S_{k+1}(x_0)$.
By sharpness of (\ref{eq:degRz}) and Lemma~\ref{l:power set}, we have $\# \bigcup_{i=1}^{k+1} A_i = k+1$. % where $A_i$ corresponds to the neighbors $y_i$ of $z$ in $S_k(x_0)$.
We define $\Phi_{k+1}:S_{k+1}(x_0) \to \iP_{k+1}([D])$, $$z \mapsto \bigcup_{i=1}^{k+1} A_i.$$ %\in \iP_{k+1}([D])$$.
Thus, the sets $A_i$ are exactly the $k$-element subsets of $\Phi_{k+1}(z)$.
I.e., for $z \in S_{k+1}(x_0)$ and $y \in S_k{(x_0)}$, we have
\begin{align}\label{eq:Phi k,k+1}
y \sim z \quad \Longleftrightarrow \quad \Phi_{\leq k}(y) \sim \Phi_{k+1}(z).
\end{align}

We define $\Phi_{\leq k+1} : B_{k+1}(x_0) \to \iP_{\leq k+1}([D])$ via
$$
x \to \begin{cases}
\Phi_{k+1}(x) &: x \in S_{k+1}(x_0) \\
\Phi_{\leq k}(x) &: x \in B_{k}(x_0). 
\end{cases}
$$
By (\ref{eq:Phi k,k+1}), we have $x \sim y \Longleftrightarrow \Phi_{\leq k+1}(x) \sim \Phi_{\leq k+1}(y)$.

It remains to show that $\Phi_{\leq k+1}$ is bijective.
To do so, it suffices to prove that $\Phi_{k+1}$ is injective since $\#S_{k+1} = \# \iP_{k+1}([D])$ and since $\Phi_{\leq k}$ is bijective and since the domains and images of $\Phi_{\leq k}$ and $\Phi_{k+1}$ are disjoint.

The idea to prove injectivity is to show that for every $x \in S_{k-1}(x_0)$, we have that every $z \in S_{2}(x)$ in the two-sphere of $x$ has exactly two backwards-neighbors w.r.t. $x$.
Then we apply the non-clustering property (NCP). From this, we will obtain injectivity of $\Phi_{k+1}$. We now give the details.

Suppose $x \in S_{k-1}(x_0)$ and $z \in S_{k+1}(x_0)$ with $d(x,z)=2$. Let $X= \Phi_{\leq k+1}(x)$ and $Z= \Phi_{\leq k+1}(z)$.
Then,
$X\subset Z$ and $\#X = k-1$ and $\#Z=k+1$. Thus, 
$\#\{Y:X\sim Y \sim Z\} = 2$, and since $\Phi_{\leq k}$ is an isomorphism, and since $\Phi_{\leq k}^{-1}(Y) \sim z$ if and only if $Y \sim Z$, we infer  $\#\{y:x\sim y \sim z\} = 2$.
I.e., for all $x \in S_{k-1}(x_0)$ and for all $z \in S_2(x) \cap  S_{k+1}(x_0)$, we have $d_-^{x}(z) = 2$. By bijectivity of $\Phi_{\leq k}$, we have for every $z \in S_2(x) \cap  B_{k}(x_0)$ that $d_-^{x}(z) = 2$. Putting these together yields
$d_-^{x}(z) = 2$ for all $z \in S_2(x)$.
We now apply (NCP) and obtain that for all $y_1,y_2 \in S_1(x)$ there is at most one $z\in S_2(x)$ with $y_1\sim z \sim y_2$.

%$B_2(x) \cong \iP_{\leq 2}([D])$.

Suppose $\Phi_{k+1}(z_1) = \Phi_{k+1}(z_2) = Z$.
Let $X \subset Z$ with $\#X = k-1$.
Then, there exist $Y_1, Y_2 \in S_1(X)$ s.t. $Z \sim Y_i$, for $i=1,2$. Let $x = \Phi_{\leq k}^{-1}(X)$ and $y_i = \Phi_{\leq k}^{-1}(Y_i)$ for $i=1,2$. Thus,
$y_i \in S_1(x)$ and $z_i \in S_2(x)$ and $y_i \sim z_j$ for $i,j=1,2$. By (NCP), we infer $z_1 = z_2$. This proves injectivity of $\Phi_{k+1}$ and hence,
$\Phi_{\leq k+1}$ is an isomorphism, completing the induction step. This finishes the proof.
%By induction, there exists a graph isomorphism $\Phi_{\leq D}: V \to \iP([D])$ which finishes the proof since $\iP([D])$ is isomorphic to a $D$-dimensional hypercube.
\end{proof}

Taking $k=D$ in the above theorem and employing the definition of the hypercube shell structure (see Definition~\ref{def:hypercube shell structure}) yields the following corollary which is the reappearance of Theorem~\ref{thm:SSL NCP hypercube INTRO}. 

\begin{corollary}\label{cor:SSL NCP hypercube}
	Let $G=(V,E)$ be a  graph with the hypercube shell structure $HSS(D,1)$. Suppose, $G$ satisfies (SSP) and (NCP). Then, $G$ is isomorphic to the $D$-dimensional hypercube.
\end{corollary}
%\begin{proof}
%	This follows immediately from Theorem~\ref{thm:local combinatorial characterization} by taking $k=D$.
%\end{proof}
%We remark that this corollary proves \ref{char:SSP NCP} $\Rightarrow$ \ref{char:hypercube} from the main theorem (Theorem~\ref{thm:main}).

%\bibliography{Bib_Liu_Muench_Peyerimhoff}

%\bibliography{Bib_Liu_Muench_Peyerimhoff}{}
%\bibliographystyle{plain}

{\bf Acknowledgements:} We gratefully acknowledge partial support by the
EPSRC Grant EP/K016687/1. FM wants to thank the German Research Foundation (DFG) and the German Academic Scholarship Foundation for financial support and the Harvard University Center of Mathematical Sciences and Applications for their hospitality.

\printbibliography
Shiping Liu, 
School of Mathematical Sciences, University of Science and Technology of China, Hefei 230026, Anhui Province, China, e-mail: 
\texttt{spliu@ustc.edu.cn}

Florentin M\"unch,
Department of Mathematics, University of Potsdam, 14476 Potsdam, Germany, e-mail:  
\texttt{chmuench@uni-potsdam.de}

Norbert Peyerimhoff, 
Department of Mathematical Sciences, Durham University, Durham DH1, 3LE, United Kingdom, e-mail:  
\texttt{norbert.peyerimhoff@durham.ac.uk}

\end{document}